\documentclass[11pt]{amsart}

\usepackage{amsmath, amssymb,  graphicx}
\usepackage[margin=1in]{geometry}

\usepackage{color}
\usepackage{equation}
\usepackage{subcaption}
\usepackage{float}
\usepackage{fancyhdr}
\makeatletter
\def\SL@eqntext#1{\rlap{\quad\SL@margintext{#1}}}

\newcommand{\mb}[1]{{\boldsymbol #1}}

\newcommand{\R}{\mathbb R}

\newtheorem{theorem}{Theorem}
\newtheorem{corollary}{Corollary}
\newtheorem{lemma}{Lemma}

\makeatletter
\renewcommand{\@endtheorem}{\endtrivlist\@endpefalse} 
\makeatother

\newcommand{\blot}[1]{}

\title{Travelling waves in a dispersion--saturating diffusion
  equation}
  
\author{Gnord Maypaokha, Nabil Bedjaoui, Joaquim M. C. Correia, Michael Grinfeld}

\begin{document}

\maketitle


\section{Introduction}

In her fundamental 1982 paper \cite{Schonbek1982} M. E. Schonbek
considered the KdV-Burgers equation, which after changes of dependent
and independent variables can be written in the form 
\begin{equation}\label{BKdV}
  u_t + f(u)_x = \epsilon u_{xx} + \delta u_{xxx}, \; t \geq 0,  \; \;
  x \in \R,
\end{equation}
with $\epsilon, \delta \geq 0$ and $f(u)=u^2/2$. She formulated
conditions on $\delta$ under which the solutions of the Cauchy problem
for (\ref{BKdV}) with appropriately chosen initial data converge in a
suitable $L^p$ space to a solution of the Cauchy problem for the
inviscid Burgers equation
\begin{equation}\label{iB}
  u_t + f(u)_x =0
\end{equation}
as $\epsilon \rightarrow 0_+$. See Theorem 4.1 of
Ref.\cite{Schonbek1982} for details. Schonbek's
paper has given rise to a considerable body of work
(see the more than 250 papers citing Ref.\cite{Schonbek1982}), in which
equations which like KdV-Burgers incorporate both diffusive and
dispersive effects have been considered with the aim of establishing
the relationship between dispersion and diffusion as the diffusion
parameter goes to zero that guarantees convergence to a solution, and
sometimes specifically to the unique entropy solution, of the Cauchy
problem for the corresponding hyperbolic conservation law.

In the analysis of convergence of solutions of the regularised
conservation laws to those of the un-regularised one, Schonbek
emphasises the importance of the existence of monotone travelling
waves; see the detailed discussion on pp. 962 and 982 of
Ref.\cite{Schonbek1982}. However, the exact connection is not clear: for
example, in Ref.\cite{PerthameRyzhik2007} the authors suggest that the
monotonicity of travelling waves is not a criterion for convergence of
all solutions of (\ref{BKdV}) to the entropy solution of
(\ref{iB}); they prove the convergence of travelling waves to shock
waves in a regime that only admits non-monotone travelling waves.
Perhaps the most that can be said at this stage is that often the fact
that dispersion and diffusion coefficients satisfy a relationship that
permits the existence of monotone travelling waves, is a sufficient
condition for proving convergence of solutions of the regularised
equations to those of the conservation law as parameters go to
zero. This is what Schonbek proved in Ref.\cite{Schonbek1982} for the
KdV-Burgers equation; in the case of a general convex flux the
situation is even less clear.

In the present paper, we consider a variation on (\ref{BKdV}) with a
generalised Rosenau-type diffusion \cite{Rosenau1990,KurganovLevyRosenau1998}: 
 \begin{equation}\label{Ros1}
  u_t+f(u)_x= \epsilon \left(\frac{u_x}{(1+u_x^2)^\alpha}\right)_x +
  \delta u_{xxx},
\end{equation}
where $f$ is a $C^2$ strictly convex function such that $f''(u)>0$ on
$\R$, and $\alpha\geq 0$. Rosenau and coauthors concentrated on the
case $\alpha=1/2$ and $\alpha=1$ with $\delta =0$. In the spirit of
Ref.\cite{Schonbek1982}, we discuss the conditions for the existence of
monotone travelling waves for (\ref{Ros1}). We call (\ref{Ros1}) the
(generalised) Rosenau-KdV equation.

We expect different behaviour for the regimes $\alpha<1/2$,
$\alpha=1/2$, and $\alpha>1/2$. The function $F$ defined by 
  \[
    F(s)= \frac{s}{(1+s^2)^\alpha},
  \]
  increases monotonically with $s$ without bound if $\alpha<1/2$ and
  is bounded and non-monotone with $\lim_{s\rightarrow \infty} F(s)=0$
  if $\alpha>1/2$; in the critical case $\alpha=1/2$, $F$ increases
  mo\-no\-to\-ni\-cal\-ly and is bounded. Hence we expect the results
  for the case $\alpha<1/2$ to be in some sense similar to the
  KdV-Burgers case, while in the case $\alpha>1/2$ the generalised
  Rosenau diffusion operator shares many properties with bounded
  operators.


The structure of the present paper is as follows: in
Section~\ref{prelim} we derive the equation for the travelling waves,
investigate its basic properties and collect the tools, mainly from
Ref.\cite{Hadeler2017}, needed for our analysis. In Section~\ref{gen} we
prove results that are true for all values of $\alpha \geq 0$; in
Section~\ref{less} we consider the parabolic-like case $\alpha<1/2$ and
in Section~\ref{more} we consider the case of $\alpha>1/2$, always
making sure to indicate what is true for the threshold classical case
$\alpha=1/2$. We conclude with remarks, suggestions for further work
and conjectures.

\section{Preliminaries} \label{prelim}

Introducing the travelling wave frame $\zeta=x-\lambda t$ and setting 
$v(\zeta)=u(x,t)$, we obtain
\begin{equation}\label{tw1}
  -\lambda v' + f(v)' = \epsilon
  \left[ \frac{v'}{\left(1+(v')^2\right)^\alpha}\right]'
  + \delta v''',
\end{equation}
where the primes denote derivatives with respect to $\zeta$.

We are looking for fronts of (\ref{Ros1}) connecting some value $u_-$
at $\zeta= -\infty$ to another value $u_+$ at $\zeta=+\infty$. As we
are interested in the creation of shocks as
$\epsilon, \delta \rightarrow 0$, we take $u_-> u_+$.

Integrating (\ref{tw1}) in $\zeta$ from $-\infty$ to $+\infty$
and assuming that all derivatives vanish as
$\zeta \rightarrow \pm \infty$, we have the Rankine-Hugoniot
condition, 
\begin{equation}\label{lambda}
  \lambda= \frac{f(u_+) -f(u_-)}{u_+ -u_-}.
\end{equation}  

Integrating now from $-\infty$ to $\zeta$, we obtain the mathematical
object we are interested in,
\begin{equation}\label{tw2}
  -\epsilon \frac{v'}{\left(1+(v')^2\right)^\alpha} = \delta v''-g(v),
\end{equation}  
where
\begin{equation}\label{gv}
  g(v) =  f(v)-f(u_-)-\lambda (v-u_-) = f(v)-f(u_+)-\lambda (v-u_+).
\end{equation}
It is easy to check that $-g(v)$ given by (\ref{gv}) is monostable and
concave and its zeros are $u_\pm$. 

We write (\ref{tw2}) as a system,
  \begin{equation}\label{twsys}
    \begin{cases}
      v'  = w, & \\
      w' = \frac{1}{\delta} \left[- \frac{\epsilon
          w}{(1+w^2)^\alpha}+g(v)
      \right]. &\\
    \end{cases}
   \end{equation} 
   Linear analysis shows that in (\ref{twsys}), $(u_-,0)$ is always a
   saddle, while $(u_+,0)$ changes from a stable spiral to a stable
   node when $\epsilon = \epsilon_0(\delta)$, where
\begin{equation}\label{e0}
  \epsilon_0 (\delta) =2\sqrt{\delta}\sqrt{-g'(u_+)}
  =2\sqrt{\delta}\sqrt{\lambda -f'(u_+)}. 
\end{equation}
Note that $\epsilon_0 (\delta)$ is independent of $\alpha$.

In the sequel we will need to understand better the flow in a
neighbourhood of the rest point $(u_+,0)$ when it is a stable node.  

Denote the eigenvalues of the system (\ref{twsys}) at the
equilibrium $(u_+,0)$ by $\chi_\pm$,
\begin{equation} \label{eigs} 
     \chi_\pm = \frac{1}{2}\left(
      -\frac{\epsilon}{\delta} \pm \sqrt{\frac{\epsilon^2}{\delta^2} +
        \frac{4 g'(u_+)}{\delta}}
    \right).
  \end{equation}
  
  We call the direction of an eigenvector corresponding to $\chi_+$
  the {\bf main direction} and the direction of an eigenvector
  corresponding to $\chi_-$ the {\bf side direction}. Note that
  Hadeler and Rothe \cite{HadelerRothe75} use the terminology {\em
    main manifold} and {\em side manifold}. It is clear that indeed
  there is an invariant 1-manifold tangent to the side direction at
  $(u_+, 0)$; in the fourth quadrant of the $(v, w)$ plane it is a
  single orbit. Proving that there is similar construct tangent to the
  main direction is more involved; see, for example Ref.\cite{KS}.

  By computing partial derivatives of $\chi_\pm$ with respect to
  $\epsilon$ and $\delta$, we have

  \begin{lemma} \label{rot}
    If $\epsilon > \epsilon_0(\delta)$, $\delta>0$, 
    \begin{equation}\label{pd}
      \frac{\partial \chi_+}{\partial \delta} < 0, \; \; \frac{\partial
        \chi_+}{\partial \epsilon} > 0,
    \end{equation}
    and 
    \begin{equation}\label{pe}
      \frac{\partial \chi_-}{\partial \delta} > 0, \; \; \frac{\partial
        \chi_-}{\partial \epsilon} < 0.
    \end{equation}
\end{lemma}

The results of Lemma~\ref{rot} mean, for example, that as we
increase $\epsilon$, the main manifold rotates anticlockwise in the
4th quadrant of the $(v, w)$ plane and becomes flatter, while the side
manifold rotates clockwise and becomes steeper, and vice versa if we
increase $\delta$. We will see the consequences of these facts below.

Similarly, for the positive eigenvalue $\theta_+$ of the linearisation
at $(u_-,0)$, by direct computation we have
\begin{lemma}\label{rot1}
  For all positive $\epsilon$ and $\delta$,
\begin{equation}\label{thp}
      \frac{\partial \theta_+}{\partial \epsilon} < 0, \; \; \frac{\partial
        \theta_+}{\partial \delta} < 0.
\end{equation}    
\end{lemma}

A very useful construction, which is extensively used in
Refs.\cite{GildingKersner2004,Hadeler2017}, is as follows: If there exists
a monotone travelling wave solution of (\ref{tw2}), we can define
$F(v)=-v'$ for $v \in [u_+,u_-]$. Then $F(u_+)=F(u_-)=0$ and
$F'(u_+) > 0$. Using the chain rule, (\ref{tw2}) transforms into
  \begin{equation}\label{feq}
      \frac{\epsilon F(v)}{\left( 1+ (F(v))^2
        \right)^\alpha} = \delta F(v) F'(v)-g(v).
   \end{equation}   
  
   To summarise, the research question addressed in this paper is then
   as follows.  Consider (\ref{tw2}). Given $\alpha$, $\delta>0$, we
   are asking what is the minimal $\epsilon$,
   $\epsilon_{min}(\alpha, \delta)$ for which this equation has {\bf
     monotone} travelling waves, i.e. monotone decreasing heteroclinic
   orbits connecting $(u_-,0)$ to $(u_+,0)$.

   We remind the reader that (\ref{tw2}) is said to be {\bf linearly
     determined} if $\epsilon_{min}(\alpha, \delta)= \epsilon_0(\delta)$
   and {\bf nonlinearly determined} if
   $\epsilon_{min}(\alpha, \delta) > \epsilon_0(\delta)$.

    The problem of existence of monotone travelling waves in the
    KdV-Burgers equation reduces to the analysis of the well-known
    (linearly determined) KPP equation; this is not the case for our
    equation (\ref{Ros1}); in our case the question of linear or
    nonlinear determinacy is far from trivial. In general, while much
    machinery invented for scalar reaction-diffusion equations and
    admirably summarised in Chapter 8 of the book of
    Hadeler, \cite{Hadeler2017} is useful for our analysis, many
    statements require additional work; in particular, Fenichel theory
    of slow-fast systems turns out to be of importance as well; good
    references for that are the papers by Hek and by Jones \cite{Hek,Jones}.

\section{Analysis for all $\alpha \geq 0$} \label{gen}

First of all, we have

\begin{theorem}\label{thex}
  For fixed $\delta$ and $\alpha$, if $\epsilon$ is large enough,
  (\ref{tw2}) has monotone travelling waves.
\end{theorem} 

\begin{proof} 
The proof follows the logic of Ref.\cite{Hadeler2017}. For
  (\ref{twsys}), consider the region $R$ in the $(v,w)$ phase plane as
  in Figure~\ref{fig1}.


  \begin{figure}[H]
    	\centerline{\includegraphics[width=0.6\textwidth]{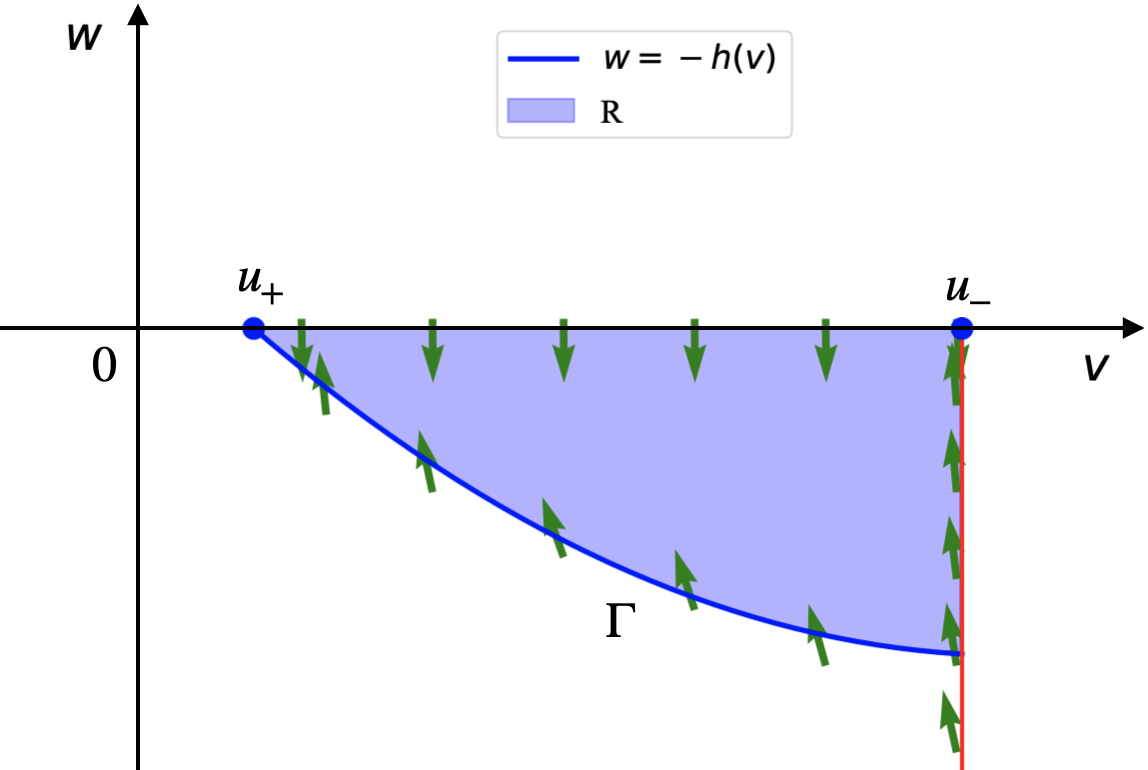}}
     	\caption{The region $R$.} \label{fig1}
  \end{figure} 

   The vector field is directed into $R$ along the horizontal
   boundary, which is a segment of the $v$ axis $[u_+,u_-]$, and the
   vertical boundary that connects the point $(u_-,0)$ with the
   intersection of the line $v=u_-$ with the graph which we denote by
   $\Gamma$, $w=-h(v)$, of some function $h$ of $v$ that belongs in
   the space
\begin{equation}\label{spaceH}
  {\mathcal H} = \{h(v) \in C^1([u_+,u_-])\, |\, h(u_+)=0, \;
  h'(v)>0, \; v \in [u_+,u_-] \}.
\end{equation}
Then $R$ will be positively invariant (and hence (\ref{twsys}) will
have a monotone orbit) if the vector field of (\ref{twsys}) along
$\Gamma$ makes an acute angle with the inward-directed normal on
$\Gamma$, i.e. if
\[
  \left(-h(v), \frac{1}{\delta} \left[ \frac{\epsilon h(v)}{\left(
          1+ h(v)^2\right)^\alpha}+g(v) \right] \right) \cdot
  (h'(v),1) \geq 0,
\]
which is equivalent to demanding that
\begin{equation}\label{ineq}
\frac{\epsilon h(v)}{(1+ h(v)^2)^\alpha}  \geq \delta
h'(v)h(v) - g(v).
\end{equation}

Now let us choose
\[
  h(v)=B(v-u_+) \in {\mathcal H}, \; \; B>0.  
\]
Thus a monotone travelling wave will exist if we can find $\epsilon$
large enough so that for some $B>0$
\[
  \frac{\epsilon}{\left(1+B^2 (v-u_+)^2\right)^\alpha} \geq
  \delta B + \frac{1}{B}\left(\lambda - \frac{f(v)-f(u_+)}{v-u_+}\right).
\]
Clearly,
\[
\frac{\epsilon}{\left(1+B^2 (v-u_+)^2\right)^\alpha} \geq
\frac{\epsilon}{\left(1+B^2 (u_--u_+)^2\right)^\alpha}.
\]
On the other hand,
\[
 \delta B + \frac{1}{B}\left(\lambda - \frac{f(v)-f(u_+)}{v-u_+}\right)
  \leq  \delta B + \frac{1}{B}\left(\lambda - f'(u_+)\right).
\]
So a monotone travelling wave will exist if there is
$B>0$ such that
\[
  \epsilon \geq  \left(1+B^2 (u_--u_+)^2\right)^\alpha
  \left[ \delta B + \frac{\lambda - f'(u_+)}{B} \right] := K(B).
\]
But the function $K(B)$ (we have omitted its dependence on
parameters that are kept fixed in this argument) is bounded from below
by some constant $K_*$, which means that if we choose
$\epsilon > K_*$, there will exist values of $B$ for which the region
$R$ is positively invariant, i.e. a monotone travelling wave solution
of (\ref{tw2}) exists.
\end{proof}

Proceeding as in Proposition 8.2 of Hadeler \cite{Hadeler2017}, we can also
prove
\begin{theorem}\label{thex1}
  For fixed $\delta$ and $\alpha$, suppose that there exists a
  monotone travelling wave for (\ref{tw2}) for some $\epsilon_1$. Then
  there exists a monotone travelling wave for all
  $\epsilon_2 > \epsilon_1$.
\end{theorem}  

\begin{proof}
  Let us denote by $w=W_1(v)$ the monotone travelling wave
  corresponding to $\epsilon_1$ and call its graph $\Gamma_1$:
  $\Gamma_1 = \{ (v, W_1(v)\, | \, v \in [u_+,u_-] \}$.  Consider the
  set $M$ in the $(v,w)$ plane bounded by the segment
  $[u_+,u_-]$ of the $v$-axis and
  $\Gamma_1$. Let us also denote by
  $\mb{f}_\epsilon$ the vector field generated by (\ref{twsys}) and by
  $\mb{n}_{\epsilon_1}$ the inward normal on $\Gamma_1$. Then
  \begin{equation}\label{ip}
    \mb{f}_{\epsilon_2} \cdot \mb{n}_{\epsilon_1} =
	  \frac{(\epsilon_2-\epsilon_1) w^2}{(1+ w^2)^\alpha} > 0,
  \end{equation}
  which means that the vector field on the curved boundary of $M$
  points into the interior of $M$. Since for $\epsilon_2 > \epsilon_1$
  by Lemma~\ref{rot1} the unstable manifold of $(u_-,0)$ for $w<0$
points into $M$, (\ref{ip}) implies that for $\epsilon_2>\epsilon_1$
there exists a heteroclinic connection between $(u_-,0)$ and $(u_+,0)$
in $M$, i.e. a monotone travelling wave.
\end{proof}  

Combining Theorem \ref{thex} and Theorem \ref{thex1} we have  that
for fixed $\alpha$ and $\delta$, (\ref{tw2}) has monotone travelling
waves iff $\epsilon$ lies in an interval of the form $(a, \infty)$ or
$[a, \infty)$ for some $a>0$.

From the proof of Theorem~\ref{thex1} we also have the
monotonicity in $\epsilon$ result:

\begin{corollary} \label{moneps1}
  Fix $\alpha$ and $\delta$. Let $w=W_1(v)$ be the monotone travelling
  wave corresponding to $\epsilon=\epsilon_1$. If $w=W_2(v)$ is
  the monotone travelling wave corresponding to
  $\epsilon=\epsilon_2> \epsilon_1$, then for all $v \in [u_+,u_-]$,
  $W_2(v) \geq W_1(v)$.
\end{corollary}

Arguing as in Proposition 8.5 of Ref.\cite{Hadeler2017} and using
Corollary~\ref{moneps1}, we have

\begin{lemma}\label{sm}
  For a fixed $\alpha$ and $\delta$, at
  $\epsilon= \epsilon_{min}(\alpha, \delta)$ (\ref{twsys}) admits a
  monotone travelling wave that approaches $(u_+,0)$ via the side
  direction and if $\epsilon > \epsilon_{min}(\delta, \alpha)$ via the
  main direction.
\end{lemma}  

As a consequence of Lemma~\ref{sm}, we have that for fixed $\alpha$,
$\delta$ the set of $\epsilon \in (0, \infty)$ for which there is
monotone travelling wave of (\ref{twsys}) is of the form
$[a,+\infty):= [\epsilon_{min}(\alpha, \delta), \infty)$, where for
fixed $\alpha$ and $\delta$, $\epsilon_{min}(\alpha, \delta)$, the
minimal $\epsilon \geq 0$ for which there exists a monotone travelling
wave for (\ref{twsys}) is well-defined by the arguments above.

We can also prove the continuity of $\epsilon_{min}(\alpha, \delta)$
with respect to $\delta$.

\begin{theorem} \label{cont}
For all $\alpha \geq 0$ the mapping
$\delta \mapsto \epsilon_{min}(\delta, \alpha)$ from $(0,\infty)$ into
itself, is continuous.
\end{theorem}

\begin{proof}
Let $\delta_0 >0$ and set $\epsilon_{min}^0 =
\epsilon_{min}(\delta_0,\alpha)$. We have to prove that
\[
  \forall \, l >0, \quad \exists \, k>0 \hbox{ such that }
  |\delta - \delta_0| \leq k \hbox{ implies }
  |\epsilon_{min}(\delta,\alpha) -\epsilon_{min}^0| \leq l.
\]

There are two cases to consider. 

Case 1: Fix $\epsilon = \epsilon_{min}^0 +l$. From Lemma~\ref{sm} it
follows that the heteroclinic orbit for $\epsilon= \epsilon_{min}^0$
approaches $(u_+,0)$ via the side direction. Since
$\epsilon > \epsilon_{min}^0$, the negative semi-orbit for any initial
condition on the side manifold crosses the line $v=u_-$ at some point
$(u_-, \tilde{w}(\delta_0))$, $\tilde{w}(\delta_0)<0$. Then by
continuity with respect to parameters, there is a $K_1>0$ such that
for all $\delta$ satisfying $|\delta- \delta_0| \leq K_1$, the side
manifold has the same behaviour. Thus there exists a monotone
travelling wave for this value of $\epsilon$ and any $\delta$
satisfying $|\delta- \delta_0| \leq K_1$, i.e.
$\epsilon_{min}(\delta, \alpha) < \epsilon = \epsilon_{min}^0 +l$.

Case 2: Fix $\epsilon = \epsilon_{min}^0 -l$. Since
$\epsilon < \epsilon_{min}^0$, the positive semi-orbit for any initial
condition on the unstable manifold of $(u_-,0)$ crosses the line
$v=u_+$ at some point $u_+,\hat{w}(\delta_0)$,
$\hat{w}(\delta_0)<0$. Then by continuity with respect to parameters,
there is a $K_2>0$ such that for all $\delta$ satisfying
$|\delta- \delta_0| \leq K_2$, the unstable manifold of $(u_-,0)$ has
the same behaviour, so that the heteroclinic orbit is not monotone.
Thus $\epsilon_{min}(\delta, \alpha) >  \epsilon = \epsilon_{min}^0
-l$ for all $|\delta- \delta_0| \leq K_2$.

Combining the two cases and choosing $K = \min (K_1, K_2)$  concludes
the proof. 
\end{proof}

Now fix $\alpha$. Then
\begin{equation}\label{epstilde}
\tilde{\epsilon}(\alpha) := \inf_{\delta \geq 0}
\epsilon_{\min}(\alpha,\delta) \geq 0,
\end{equation}
is well defined by Theorem \ref{cont}. We will give an estimate this quantity
in~\ref{epst}.

We also have the following monotonicity in $\alpha$ result:

\begin{theorem}\label{theoMonAlpha}
  Let us be given $0\leq \alpha_1 <\alpha_2$. Then for any
  $\delta >0$, 
  
  \noindent
  $\epsilon_{min} (\alpha_1, \delta)~\leq~\epsilon_{min} (\alpha_2, \delta).$
 \end{theorem} 

\begin{proof}
   Fix $\delta$, and let $\alpha_1$, $\alpha_2$, $\epsilon$ satisfy
   $0\leq \alpha_1 <\alpha_2$ and
   $\epsilon >\epsilon_{min}(\alpha_2, \delta)$. Thus the travelling
   wave $v\mapsto W_2(v)$ of (\ref{tw2}) for $\alpha = \alpha_2$ is
   strictly monotone decreasing.  Let us also consider the trajectory
   for an initial condition on the unstable manifold of $(u_-,0)$
   for $\alpha = \alpha_1$. It is defined on the interval
   $[u_+, u_-]$ by a function $v\mapsto W_1(v)$. We aim to show that
   $W_1(u_+)=0$. We have
    \[
      \delta W_1 \frac{d W_1}{dv}(v) + \epsilon \frac{W_1(v)}{(1 +
        W_1^2(v))^{\alpha_1}} = g(v) = \delta W_2 \frac{d W_2}{dv}(v) +
      \epsilon \frac{W_2(v)}{(1 + W_2^2(v))^{\alpha_2}}.
 \]
Setting $Z_i = W^2_i$, we have
 \[
   \frac{d Z_1}{d v} - \frac{d Z_2}{d v} =   \frac{2\epsilon}{\delta}
   \left(  \frac{\sqrt{Z_1}}{(1 + Z_1)^{\alpha_1}} 
   -\frac{\sqrt{Z_2}}{(1 + Z_2)^{\alpha_2}}\right).
\] 
But $\alpha_1 <\alpha_2$, so we can write
\[
  \frac{d Z_1}{d v} - \frac{d Z_2}{d v} \geq  \frac{2\epsilon}{\delta}\left(  \frac{\sqrt{Z_1}}{(1 + Z_1)^{\alpha_2}} 
    -\frac{\sqrt{Z_2}}{(1 + Z_2)^{\alpha_2}}\right).
\]  
Defining the function  $z\mapsto  l(z) :=\frac{\sqrt{z}}{(1 +
  z)^{\alpha_2}}$ which is monotone increasing for small $z$, we can write
\[ 
 \frac{d Z_1}{d v} - \frac{d Z_2}{d v} \geq  \frac{2\epsilon}{\delta}
 \frac{d\,l}{dz} (z_{1,2}(v)) \left(Z_1- Z_2\right),
\] 
where $z_{1,2}(v)$ is some value between $Z_1(v)$ and $Z_2(v)$.
Define
\[ 
  L(v)= \frac{2\epsilon}{\delta} \frac{d\,l}{dz} (z_{1,2}(v)).
\]
The function $L$ is defined on $(u_+, u_-)$. Also, since $l$ is
monotone increasing in the neighbourhood of $0$ and the mappings
$v\mapsto Z_1(v)$ and $v\mapsto Z_2(v)$ are continuous at $v=u_-$, and
satisfy $Z_1(u_-) = Z_2(u_-)=0$, we obtain that $L(v)\geq 0$ in an
interval of the form $(\overline{u}, u_-)$, $\overline{u}\in (u_+, u_-)$. 
The function $L(v)$ can only go to infinity as $v \rightarrow
u_\pm$. This function is not integrable in any sub-interval of
$[u_+,u_-]$ containing $u_-$, as in a neighbourhood of $u_-$ we have
\[
  L(v) \sim \frac{\epsilon}{\delta\theta_+|v-u_-|}.
  \]
Hence $\lim_{u \rightarrow (u_-)_-} \int_{\overline{u}}^u L(u) \, du =
\infty$.
We have 
\[
 \frac{d}{d v}\left( (Z_1(v)- Z_2(v)) e^{-\int_{\overline{u}}^v L(u) du}
 \right) \geq 0.
\]
 
Integrating this formula over $[v, u]$ for any $v$ and $u$ in
$(u_+, u_-)$, $v < u$, and taking the limit as
$u \rightarrow (u_-)_-$, we have that $ Z_1(v) -Z_2(v) \leq 0$ so that
$W_1(v) \geq W_2(v)$ for all $v\in ( u_+, u_-)$. Thus, necessarily
$W_1(u_+) = 0$ and the corresponding heteroclinic travelling wave is
monotone, which means that
$\epsilon \geq \epsilon_{min}(\alpha_1, \delta)$ and hence
$\epsilon_{min} (\alpha_1, \delta)~\leq~\epsilon_{min} (\alpha_2,
\delta)$.
\end{proof}

Using arguments similar to those in Theorems~\ref{thex}
and~\ref{thex1} but with a different choice of function $h(v)$, we can
also prove a linear determinacy result for large enough $\delta$.

\begin{theorem}\label{theolin}
  For any $\alpha >0$, there exists  $\delta_{\alpha}< \infty$, such
  that for any $\delta\geq \delta_{\alpha}$,
  \[
    \epsilon_{min} (\alpha,\delta) = \epsilon_0(\delta)
  \]
and the map $\alpha \mapsto \delta_\alpha$ is monotone increasing.   
\end{theorem}

\begin{proof}
  Fix
  $\epsilon \geq \epsilon_0 (\delta) = 2
  \sqrt{\delta}\sqrt{-g'(u_+)}$. We aim to show that for
  $\delta \geq \delta_\alpha$,
  $\epsilon~\geq~\epsilon_{min}(\alpha, \delta)$.
  
  Let us consider the system (\ref{twsys}). The positive eigenvalue
  $\theta_+$ of the linearisation at $(u_-, 0)$ is
\[  
  \theta_+ = \frac{1}{2}\left( -\frac{\epsilon}{\delta} +
    \sqrt{\frac{\epsilon^2}{\delta^2} + \frac{4
        g'(u_-)}{\delta}}\right) = \frac{2 g'(u_-)}{\epsilon +
    \sqrt{\epsilon^2 + 4 \delta g'(u_-)}}.
\]
Consider now the function 
\[
  h (v) = -\frac{1}{\sqrt{\delta}\sqrt{-g'(u_+)}} g(v).
\]  


Then $-h (v)$ satisfies
\[
  -h'(u_-) = \frac{1}{\sqrt{\delta}\sqrt{-g'(u_+)}} g'(u_-).
\]
But for  $\epsilon\geq \epsilon_0(\delta)$, we have
\[
  \theta_+ <  \frac{2 g'(u_-)}{\epsilon} \leq -h'(u_-).
\]  
Thus, the graph of $-h$ is below the heteroclinic orbit (not
necessarily monotone at this stage) in a neighbourhood of $(u_-,0)$. If
the condition (\ref{ineq}) is satisfied, this travelling wave will
remain above the graph of $-h$ and thus will indeed be
monotone. (\ref{ineq}) can be written as
\[
\frac{\epsilon}{\sqrt{\delta}} \geq K(v) := \left( -
  \frac{1}{\sqrt{-g'(u_+)}} g'(v)+ \sqrt{-g'(u_+)} \right) \left( 1  -
  \frac{1}{\delta g'(u_+) } g^2(v) \right)^\alpha,
\]
for all $v\in [u_+, u_-]$. We have

\begin{align*} 
& K'(v) = - \frac{1}{\sqrt{-g'(u_+)}} g''(v) \left( 1  - \frac{1}{\delta
    g'(u_+) } g^2(v) \right)^\alpha\\
&-\frac{2 \alpha}{\delta g'(u_+) } g(v) g'(v) \left( - \frac{1}{\sqrt{-g'(u_+)}} g'(v)+ \sqrt{-g'(u_+)} \right)
\left( 1  - \frac{1}{\delta g'(u_+) } g^2(v) \right)^{\alpha-1}.
\end{align*}

But $g''(v) =f''(v) \geq C > 0$ for all $v\in [u_+, u_-]$. Thus for $\delta$
sufficiently large, 
\[
  K'(v) \leq  - \frac{C}{2\sqrt{-g'(u_+)}} <0.
\]  
It follows that $K(v) \leq K(u_+) = 2 {\sqrt{-g'(u_+)}}$ for $\delta$
sufficiently large, and hence for all $\delta \geq \delta_\alpha$ and
$\epsilon \geq \epsilon_0(\delta)$, the wave is monotone and hence
$\epsilon_{min}(\alpha,\delta) \leq \epsilon_0(\delta)$, but then
these two must be equal.

The last part of the theorem follows directly from
Theorem~\ref{theoMonAlpha}.
\end{proof}

An estimate of  $\delta_\alpha$ will be given in \ref{epst}.

\noindent
{\bf Remark:}  When $\alpha =0$, $K'(v) <0$ for any
$\delta >0$, and thus we recover that
$\epsilon_{min}(0,\delta) = \epsilon_{0}(\delta)$ for all $\delta >0$ in
the KPP case.

Using geometric singular perturbation theory, \cite{Hek,Jones} we have
the following existence result for small values of $\delta$, which
allows us to characterise $\epsilon_{min}(\alpha, \delta)$ for all
$\alpha \geq 0$.

\begin{theorem}\label{smdelta}
  Assume $\alpha \geq 0$ and let $\epsilon_{min}(\alpha)$ be given
  by~(\ref{emin2}).  Fix $\epsilon$ such that
  \begin{equation}\label{case1}
    \epsilon > \begin{cases}
      \epsilon_{min}(\alpha) & if $\quad \alpha > 1/2$,\\
      S= \lim_{\alpha \rightarrow (1/2)_+} \epsilon_{min}(\alpha) &
      if $\quad \alpha=1/2$\\
      0 & if $\quad \alpha< 1/2$.
    \end{cases}
\end{equation}
    
 where $S$ is defined by~(\ref{S}),  or
  
    \begin{equation}\label{case2}
  \epsilon =  \epsilon_{min}(\alpha) \quad \text{if}\quad  \alpha > 1/2.
    \end{equation}
  
  Then  there exists $\delta_0(\alpha,\epsilon) >0$ such that for all
  $0 < \delta < \delta_0(\alpha,\epsilon)$

  (i) (\ref{twsys}) admits a monotone
  travelling wave;

   (ii)  the monotone
   travelling wave enters $(u_+,0)$ through the main direction.
 \end{theorem}

\begin{proof} 
Assume that we have (\ref{case1}). Write (\ref{twsys}) equivalently as
  \begin{equation}\label{sp}
    \begin{cases}
      v'  = w, & \\
     \delta w' = \left[- \frac{\epsilon
          w}{(1+w^2)^\alpha}+g(v)
      \right]. &\\
    \end{cases}
   \end{equation} 
(i) Note that in each one of the cases of the theorem, for $\delta=0$, the
right-hand side of the second equation defines $w$ as a function of of
$v$, say $w=W_+(v)$ and $W_+(u_\pm)=0$. This means that for $\epsilon$
satisfying the conditions of the theorem, there exists a smooth
monotone travelling wave of (\ref{sp}) for $\delta=0$. But then we
can appeal to geometric singular perturbation theory \cite{Hek,Jones}
and, using Fenichel's First Theorem in the graph form as in
Theorem 2 of Ref.\cite{Jones}, deduce the existence of monotone travelling
waves for $\delta$ sufficiently small.

(ii) $C^1$ closeness in $\delta$ of the graph of $W_+(v)$ and the
monotone travelling wave solutions of (\ref{twsys}) for $\delta>0$
small is guaranteed by Theorem 2 of Ref.\cite{Jones}. However,
  $W_+'(u_+)$ is finite while $\lim_{\delta \rightarrow 0_+}
  \chi_-(\delta)= -\infty$, from which it follows that for
  sufficiently small $\delta>0$ the monotone travelling wave cannot
  enter via the side manifold.

Now, let us consider the singular case (\ref{case2}), i. e.
$\epsilon=\epsilon_{min}(\alpha)$ with $\alpha>1/2$. It
would be interesting to prove this statement using blowup techniques
of, for example, Krupa and Szmolyan, \cite{KS2001} but we find it
easier to argue directly by using properties of the vector field. 

Note that if $\alpha>1/2$, $\epsilon = \epsilon_{min}(\alpha)$, the
curves $w=W_\pm(v)$ used in the proof of the case (\ref{case1}) are
still defined, though now with $W_+(v) \geq W_-(v)$ for all
$V \in [u_+,u_-]$.  But we have $W_+(v)=W_-(v) := \bar{w}$ at the
point $\bar{u}$ defined by $g(\bar{u}) = -S$. This is the point of
global minimum for $W_+(v)$.  Note also that this curve has a corner at the point $(\bar{u}, \bar{w})$. We have
\[
  W_+'(u_+) = \frac{g'(u_+)}{\epsilon}.
\]
To prove existence of monotone heteroclinic connections between
$(u_-,0)$ and $(u_+,0)$ we use the logic of Theorem~\ref{thex}. 
Thus, we need to show that there exists $B>0$ such that
\vspace{-0.2cm}
\begin{equation}\label{ineqB}
  \frac{\epsilon}{(1+ B^2 (V-u_+)^2)^\alpha} > 
  \delta B +  \frac{1}{B}\left(\lambda - \frac{f(V)-f(u_+)}{V-u_+}\right) =
  \delta B -  \frac{g(V)}{B(V-u_+)}.
\end{equation}
First of all, let us choose $B$. Take
\begin{equation}\label{cB}
  B = -2 W_+'(u_+) = \frac{-2g'(u_+)}{\epsilon}.
\end{equation}
Now, if $v=u_+$, if
\[
  \delta \leq \tilde{\delta}= - \frac{\epsilon^2}{8 g'(u_+)},
\]
we have that the inequality (\ref{ineqB}) holds. Hence by continuity
there is $v_1 \in [ u_+, u_-]$ such that (\ref{ineqB}) holds for all $
v \in [u_+, v_1]$ and that in  addition $h(v)=B(u_+-v)< W_+(v)$ for
all $v$ in that interval.

Setting $v_2= (u_++ v_1)/2$,  we define the compact set
\begin{equation}\label{K12}
  K_{12}= \{ (v,w) \, | \, v_2 \leq v \leq v_1, \; \bar{w} \leq w \leq
  h(v) \}, 
\end{equation}
as in Figure \ref{fig2}. Let
$d_\epsilon(v,w) = {\displaystyle g(v) - \frac{\epsilon
    w}{(1+w^2)^\alpha}}$, so that the union of graphs of $w=W_\pm(v)$
corresponds to the set $d_\epsilon(v,w)=0$. As $K_{12}$ is disjoint
from the graph of $W_+(v)$ and on it $d_\epsilon(v,w)>0$, by
compactness we have that there exists $C_{12} >0$ such that for all
$(v,w) \in K_{12}$, 

\[
  d_\epsilon (v,w) \geq \min_{(v,w) \in K_{12}} d_\epsilon(v,w) =
  C_{12}.
\]
Now we can define our value of $\delta_0(\alpha, \epsilon)$.
Let
\begin{equation}\label{del0}
  \delta_0(\alpha, \epsilon) := \min \left( \tilde{\delta},
    \frac{2C_{12}(v_1-v_2)}{\bar{w}^2} \right) > 0.
\end{equation}
We have
\begin{lemma}\label{Ldel0}
  For all $0< \delta \leq \delta_0(\alpha,\epsilon)$, there is
  $v_3 \in [v_2,v_1]$ such that the segment of the heteroclinic orbit 
  connecting $(u_+,0)$ and $(u_-,0)$ lying $[u_+,u_-]$, which we
  denote by $W(v)$, satisfies $W(v_3) > h(v_3)$
  (i.e. at that point it lies above the set $K_{12}$).
\end{lemma}

\begin{proof}
  First of all note that by Lemma~\ref{rot1}, for $\delta>0$ the
  heteroclinic orbit connecting $(u_+,0)$ and $(u_-,0)$ crosses the
  graph of $W_+(v)$ at a value of $v < \bar{u}$ and $W_+'(v)$ is
  negative for all $v < \bar{u}$.

  Now assume by contradiction that $W(v)$ stays in $K_{12}$ for all
  $v \in [v_1,v_2]$. for $v$ in that interval, $W(v)$ satisfies

  \[
    W \frac{dW}{dv} = \frac{d_\epsilon(v,{W})}{\delta},
  \]
  integrating which over $[v_1, v_2]$ we have
  \[
    \frac12 (W(v_1)^2-W(v_2)^2) \geq \frac{C_{12}(v_1-v_2)}{\delta}
    \geq \frac{\bar{w}^2}{2},
   \]
   by definition of $\delta_0(\alpha,\epsilon)$. This inequality however is
   impossible by the definition of $\bar{w}$. 
\end{proof}
\begin{figure}[H]
    	\centerline{\includegraphics[width=0.55\textwidth]{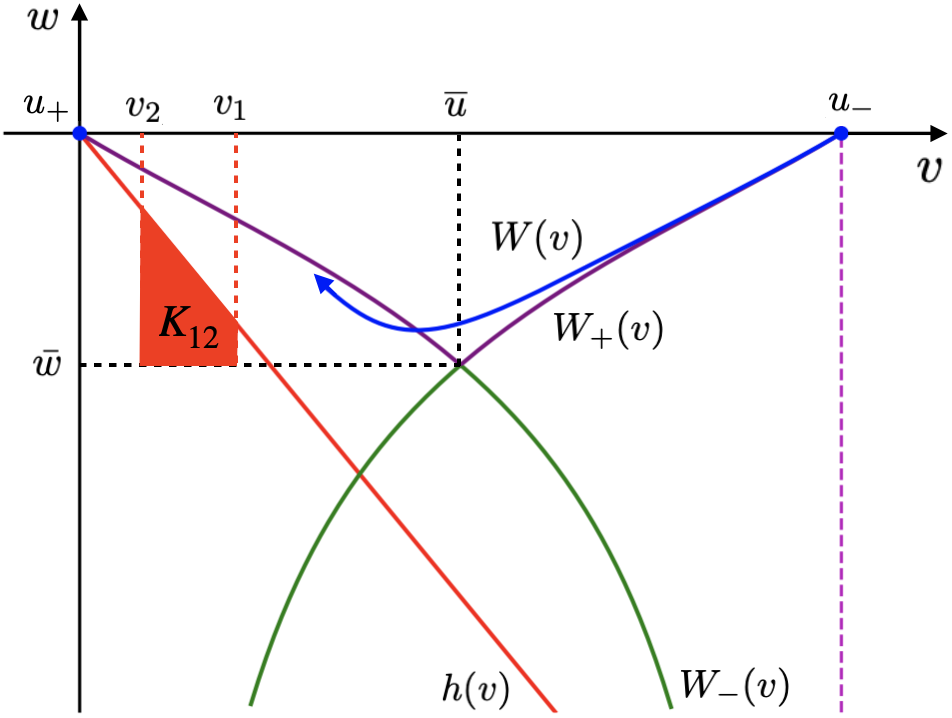}}
     	\caption{Construction of the set $K_{12}$.} \label{fig2}
\end{figure} 

From Lemma~\ref{Ldel0}, the inequality (\ref{ineqB}) therefore holds
at $v_3$ and hence for all $v~\in~[u_+, v_3]$ and therefore $W(v)$
converges to $(u_+,0)$ monotonically.  

Now, for $\delta \leq \delta_0(\alpha,\epsilon)$,  
$$\chi_-(\delta) \leq -\frac{\epsilon}{ 2 \delta} \leq -\frac{\epsilon}{ 2 \tilde{\delta}} \leq \frac{4 g'(u_+)}{\epsilon} < -B.$$

So the slope of $h(v)$ is greater
than $\chi_-(\delta)$, and hence the convergence is via the main direction. 
\end{proof}

\noindent {\bf Remark:} The method used for (\ref{case2}) is of course
applicable in the case (\ref{case1}).
\section{Results for $\alpha \in (0, 1/2]$} \label{less}

First of all, we have the following result

\begin{lemma} \label{c0}
  For all $\alpha \leq 1/2$ we have that
  \begin{equation}\label{est}
    \epsilon_{min}(\alpha,\delta) \geq \delta^{1/2-\alpha} C_\alpha,
  \end{equation}
  where $C_\alpha>0$.
\end{lemma}

\begin{proof}
  If a monotone travelling wave exists, from (\ref{feq}), we have that
\[
   \epsilon F(v)^{1-2\alpha} \geq \epsilon
   \frac{F(v)}{(1+F(v)^2)^\alpha} = \delta F(v)F'(v) -g(v).
 \]
 Hence
 \begin{equation} \label{e00}
   \epsilon \geq \frac{\delta}{2\alpha+1} \left( F(v)^{2\alpha+1}
   \right)' -g(v) F(v)^{2\alpha-1}.
 \end{equation}
 Integrating over $[u_+.u_-]$, we obtain
 \begin{equation} \label{e1}
   \epsilon (u_--u_+) \geq \int_{u_+}^{u_-} |g(v)| F(v)^{2\alpha-1} \,
   dv.
 \end{equation}
 On the other hand, integrating (\ref{feq}) over $[v, u_-]$, $v \in
 [u_+,u_-]$, we have
 \[
  -\delta \frac{F(v)^2}{2} = \int_v^{u_-} g(s) \, ds +
   \epsilon\int_v^{u_-} \frac{F(s)}{(1+F(s)^2)^\alpha} \, ds \geq 
\int_v^{u_-} g(s) \, ds.
 \]
 Hence
 \[
   F(v)^2 \leq \frac{2}{\delta} \int_{v}^{u_-} |g(s)| \, ds
 \]
 and therefore if we set $K=\int_{u_+}^{u_-} |g(s)| \, ds$, we have
   \begin{equation} \label{Fbound}
     F(v) \leq \left( \frac{2K}{\delta} \right)^{1/2}.
   \end{equation}
Substituting the bound (\ref{Fbound})  in (\ref{e1}), we obtain for
{$\alpha \leq1/2$}, that if a monotone travelling wave exists, 
\[
  \epsilon \geq \delta^{1/2-\alpha} \frac{2^{\alpha-1/2}}{(u_-- u_+)} K^{\alpha+1/2}.
\]
Therefore we have the result of the Lemma with
\[
  C_\alpha = \frac{2^{\alpha-1/2}}{(u_-- u_+)} K^{\alpha+1/2}.
\]  
\end{proof} 

\noindent 
{\bf Remark:} If $\alpha=1/2$, from the last inequality we get
\[
  \epsilon_{min}(1/2, \delta) \geq \frac{1}{(u_- -u_+)} \int_{u_+}^{u_-}
|g(s)| \, ds.
\]
However, a better estimate is obtained by integrating (\ref{e00}) over $[u_+,v]$ that is 
\[
  \epsilon_{min}(1/2, \delta) \geq \max_{v\in[u_+,u_-]}\left(\frac{1}{(v -u_+)} \int_{u_+}^{v}
|g(s)| \, ds,\right)
\]

\noindent 
which in the case of $g(v)=v(v-2)/2$, $[u_+,u_-] = [0,2]$, gives
$\epsilon_{min}(1/2, \delta) \geq 3/8$, see Figure \ref{fig4} (Figure \ref{fig3} illustrates the case $\alpha < 1/2$). This value looks optimal since
it corresponds to the limit of $\epsilon_{min}(1/2, \delta)$ as
$\delta$ goes to zero (see~\ref{three}).

\begin{figure}[H]
  	\centering
  	\begin{subfigure}{0.46\textwidth}
   		\vspace{0.25cm}
    		\includegraphics[width=\textwidth]{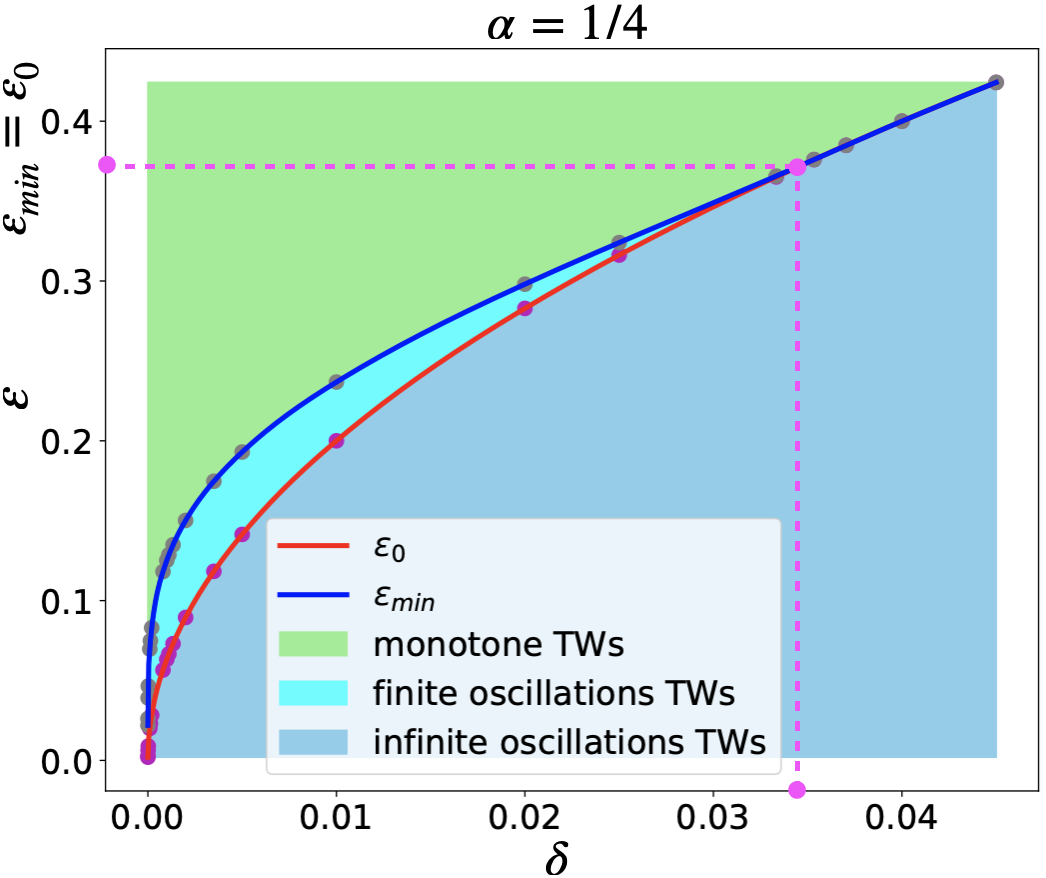}
  	\end{subfigure}
  	\hfill
  	\begin{subfigure}{0.525\textwidth}
   	 	\includegraphics[width=\textwidth]{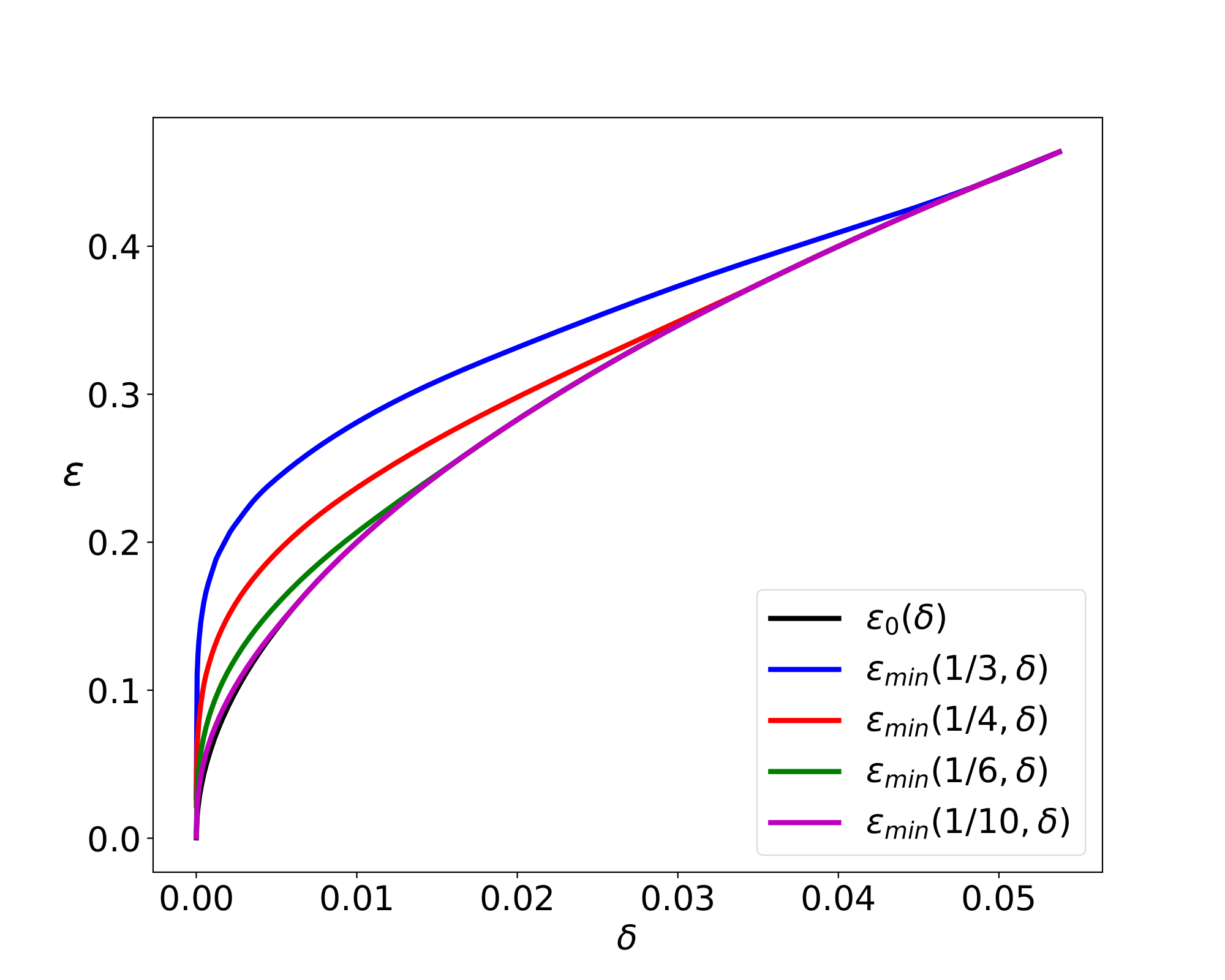}
  	\end{subfigure}
  	\caption{$\epsilon_{min}(\alpha, \delta)$ for $\alpha <1/2$.}\label{fig3}
\end{figure}
 


\begin{figure}[H]
    	\centerline{\includegraphics[width=0.55\textwidth]{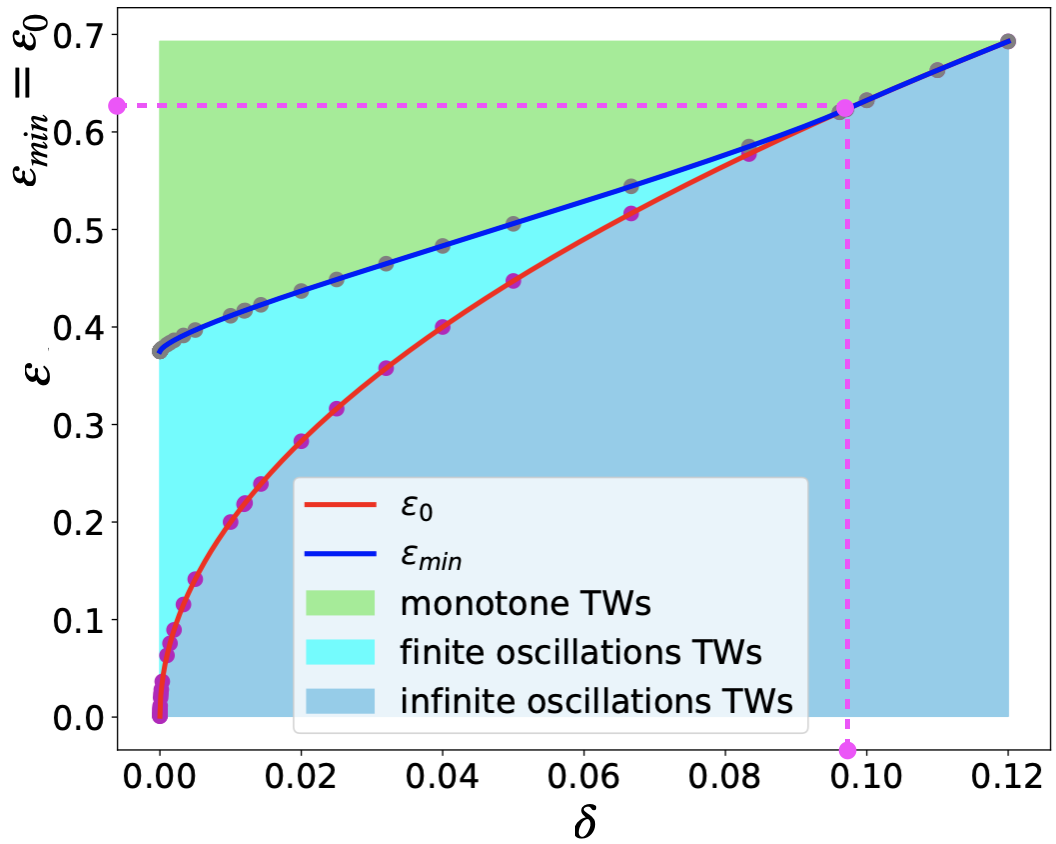}}
     	\caption{$\epsilon_{min}(\frac{1}{2}, \delta)$.} \label{fig4}
\end{figure} 

\noindent
We immediately obtain a non-linear determinacy result:

\begin{lemma}\label{nlinLess}
  For $\alpha \leq 1/2$, (\ref{twsys}) is nonlinearly determined for
  $\delta>0$ small enough. 
\end{lemma}

The result follows directly from Lemma~\ref{c0}.

We also have the following monotonicity result: 

\begin{theorem}\label{theoMondelta}
  Let $0\leq \alpha \leq 1/2$ fixed. Then the function
  $\delta\mapsto\epsilon_{min} (\alpha, \delta)$ is strictly monotone
  increasing.
 \end{theorem} 

\begin{proof}
  Using the results of Lemma~\ref{rot}, we have that if
  $0< \delta_1 < \delta_2 \leq -\frac{\epsilon^2}{4 g'(u_+)}$, we have
  that
\begin{equation}\label{eigs1}
  \xi_-(\delta_1) < \xi_-(\delta_2)< \xi_+(\delta_2)< \xi_+(\delta_1).
\end{equation}  

Now assume by contradiction that
$\epsilon_{min} (\alpha, \delta_1) \geq \epsilon_{min} (\alpha,
\delta_2),$ and take $\epsilon = \epsilon_{min} (\alpha,
\delta_1)$. By (\ref{eigs1}) and the arguments of Proposition
8.5 of Ref.\cite{Hadeler2017}, there is a monotone travelling wave, say
$v\mapsto W_1(v)$, connecting $(u_+,0)$ to $(u_-,0)$ with slope
$X_-(\delta_1)$ at $(u_+,0)$, and a monotone travelling wave
$v\mapsto W_2(v)$ connecting the same equilibria with slope
$X_-(\delta_2)$ ( if
$\epsilon_{min} (\alpha, \delta_1)= \epsilon_{min} (\alpha,
\delta_2)$) or $X_+(\delta_2)$ (if
$\epsilon_{min} (\alpha, \delta_1) > \epsilon_{min} (\alpha,
\delta_2)$). Thus $W_2(v) >W_1(v)$ in some interval $(u_+, \tilde{u})$
where $\tilde{u} \leq u_-$ and $W_1(\tilde{u}) = W_2(\tilde{u})$.

We have
\[
  \delta_1 W_1 \frac{d W_1}{dv}(v) + \epsilon \frac{W_1(v)}{(1 +
    W_1^2)^{\alpha}} = g(v) = \delta_2 W_2 \frac{d W_2}{dv}(v) +
  \epsilon \frac{W_2(v)}{(1 + W_2^2)^{\alpha}}.
\] 
 
Integrating this expression in $v$ over $[u_+, \tilde{u}]$, we have that
\[
  \frac{1}{2}(\delta_1-\delta_2)W_2(\tilde{u})^2 = \int_{u_+}^{\tilde{u}}  \frac{W_2(v)}{(1 + W_2^2)^{\alpha}}dv-
  \int_{u_+}^{\tilde{u}}  \frac{W_1(v)}{(1 + W_1^2)^{\alpha}}dv.
\]  
Since $\alpha \leq 1/2$, the function
$x\mapsto \frac{x}{(1+x^2)^\alpha}$ is strictly monotone increasing
and thus we get a contradiction, since
$ \int_{u_+}^{\tilde{u}} \frac{W_2(v)}{(1 + W_2^2)^{\alpha}}dv-
\int_{u_+}^{\tilde{u}} \frac{W_1(v)}{(1 + W_1^2)^{\alpha}}dv>0,$ when
$\frac{1}{2}(\delta_1 -\delta_2)W_2(\tilde{u})^2 \leq 0$.
\end{proof}

If $\alpha< 1/2$, we can also make a statement about the limit of
$\epsilon_{min}(\alpha,\delta)$ as $\delta \rightarrow 0_+$.

\begin{lemma} \label{del00} For $\alpha< 1/2$,
  $\lim_{\delta \rightarrow 0_+} \epsilon_{min}(\alpha,\delta)=0$.
\end{lemma}

\begin{proof}
  This follows from Theorem~\ref{smdelta} in the case of $\alpha<1/2$
  and Theorem~\ref{theoMondelta}.
\end{proof}

Thus in the case of $\alpha< 1/2$ the description of
$\epsilon_{min}(\alpha, \delta)$ is complete.

\section{Results for $\alpha > 1/2$} \label{more}

For $\alpha \geq  1/2$, we define $ \epsilon_{min}(\alpha)$ to be the
minimum value of $\epsilon$ for which (\ref{tw2}) with $\delta=0$ has
a smooth monotone travelling wave.

It turns out that this quantity can be computed exactly.

\begin{lemma}\label{epsmin}
For all $\alpha > 1/2$ we have  
 \begin{equation}\label{emin2}
   \epsilon_{min}(\alpha) = H(\alpha) := 2^\alpha S
   \left(\frac{\alpha}{2\alpha-1} \right)^\alpha \sqrt{2\alpha-1},
 \end{equation}  
where
\begin{equation}\label{S}
  S = \max_{v \in [u_+,u_-]} [-g(v)].
\end{equation}
\end{lemma}

\begin{proof}
  If we set $\delta=0$ in (\ref{tw2}), we have that the existence of
  monotone travelling waves is equivalent to the solvability of
\begin{equation} \label{tw3}
  -\epsilon v' = -g(v) \left(1+ (v')^2 \right)^\alpha, \; \; v \in
  [u_+,u_-],
\end{equation}
with $v'<0$ in $[u_+,u_-]$. But this condition is equivalent to having
intersections between the graphs of
\[
  h_1(z)=-\epsilon z \hbox{ and }
  h_2(z)= S(1+z^2)^\alpha,
\]
$z<0$. Therefore $\epsilon_{min}(\alpha)$ is determined by the
tangency of these two graphs, i.e. from the system of equations
$h_1(z)=h_2(z)$ and $h_1'(z)=h_2'(z)$, that is, 
\begin{equation}\label{emin}
  \begin{cases}
    -\epsilon_{min}(\alpha) z = S(1+z^2)^\alpha, & \\
    -\epsilon_{min}(\alpha) = 2\alpha zS(1+z^2)^{\alpha-1}. &
   \end{cases} 
 \end{equation}
 From the second of equations (\ref{emin}) we have
 \[
   -\epsilon_{min}(\alpha) z= \frac{2\alpha
     z^2S(1+z^2)^{\alpha}}{1+z^2}.
 \]
 Hence using the first equation in (\ref{emin}), we obtain
 \[
  S(1+z^2)^\alpha= \frac{2\alpha
    z^2S(1+z^2)^{\alpha}}{1+z^2}.
 \] 
 Therefore $2\alpha z^2=1+z^2$ and hence at the point of tangency,
 requiring $z<0$, we have
 \[
   z = - \frac{1}{\sqrt{2\alpha-1}}.
 \]
 Substituting this information in (\ref{emin}), we have the result of
 the Lemma. 
\end{proof}
 
It is not hard to show that
$\lim_{\alpha \rightarrow 1/2_+} H(\alpha)=S$, a result to which we
will return, and that the asymptotics of $H(\alpha)$ as
$\alpha \rightarrow \infty$ is given by
\[
  H(\alpha) \sim \sqrt{2e} \sqrt{\alpha} + O(1/\sqrt{\alpha}).
\]

\noindent \noindent {\bf Remark:} It is rather surprising that Hadeler--Rothe type
approximation leads exactly to the formula (\ref{emin2}). Please
see~\ref{HR}.

Next we want to characterise graph of the mapping $\delta \mapsto
\epsilon_{min}(\alpha,\delta)$ if $\alpha>1/2$.

First of all note that for $\alpha>1/2$,
$\tilde{\epsilon}(\alpha)$ defined in (\ref{epstilde}) must satisfy
$\tilde{\epsilon}(\alpha)~<~\epsilon_{min}(\alpha)$.

We have the following Corollary of Theorem~\ref{smdelta}(ii):

\begin{corollary}\label{epsineq}
  For $\alpha > 1/2$,
  $\tilde{\epsilon}(\alpha)={\inf}_{\delta \in \R_+} \epsilon_{min}(\alpha,\delta) <
 \epsilon_{min}(\alpha)$.
\end{corollary}

This follows since by Lemma~\ref{sm} at
$(\delta, \epsilon_{min}(\alpha,\delta))$ the monotone travelling wave
approaches $(u_+,0)$ via the side manifold.  

\noindent 
{\bf Remark:} This means that for $\alpha>1/2$,
$\epsilon_{min}(\alpha,\delta)$ has at least one minimum
$\tilde{\epsilon}(\alpha)<\epsilon_{min}(\alpha)$, see Figure \ref{fig5}. To prove that
uniqueness of the minimum seems like an interesting problem.

\begin{figure}[H]
  	\centering
  	\begin{subfigure}{0.46\textwidth}
   		\vspace{0.25cm}
    		\includegraphics[width=\textwidth]{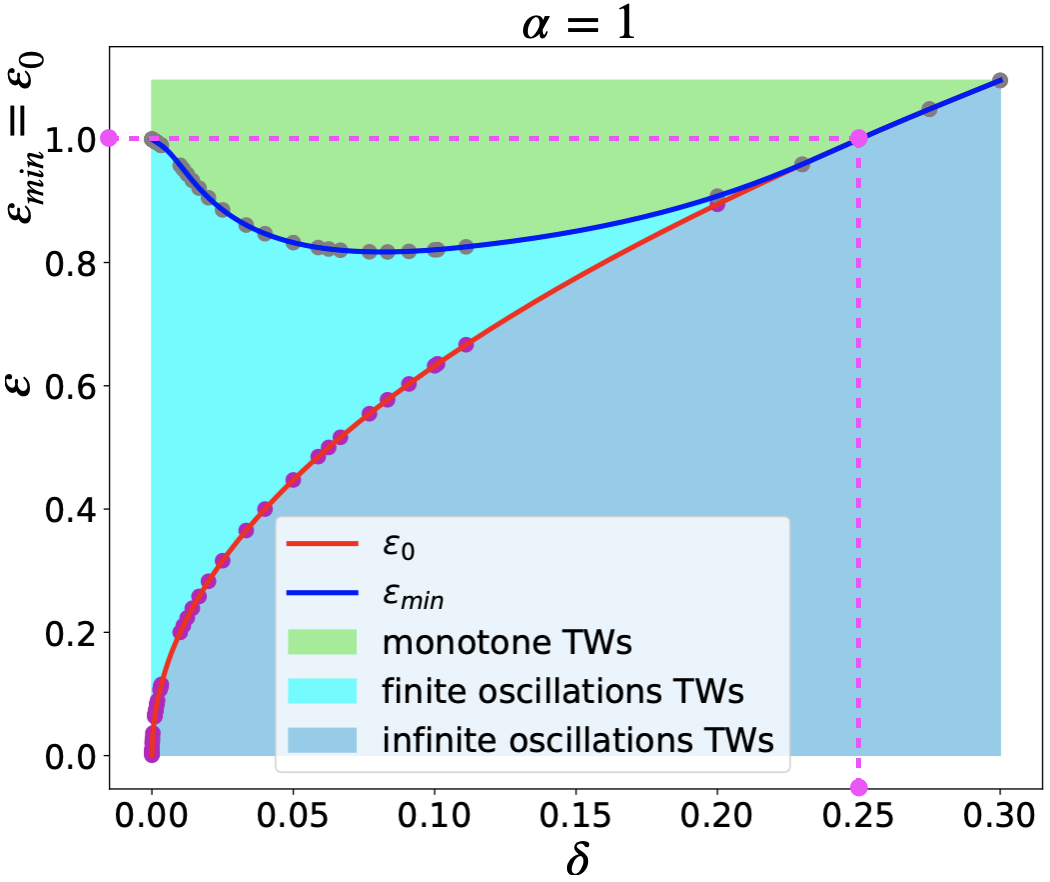}
  	\end{subfigure}
  	\hfill
  	\begin{subfigure}{0.525\textwidth}
   	 	\includegraphics[width=\textwidth]{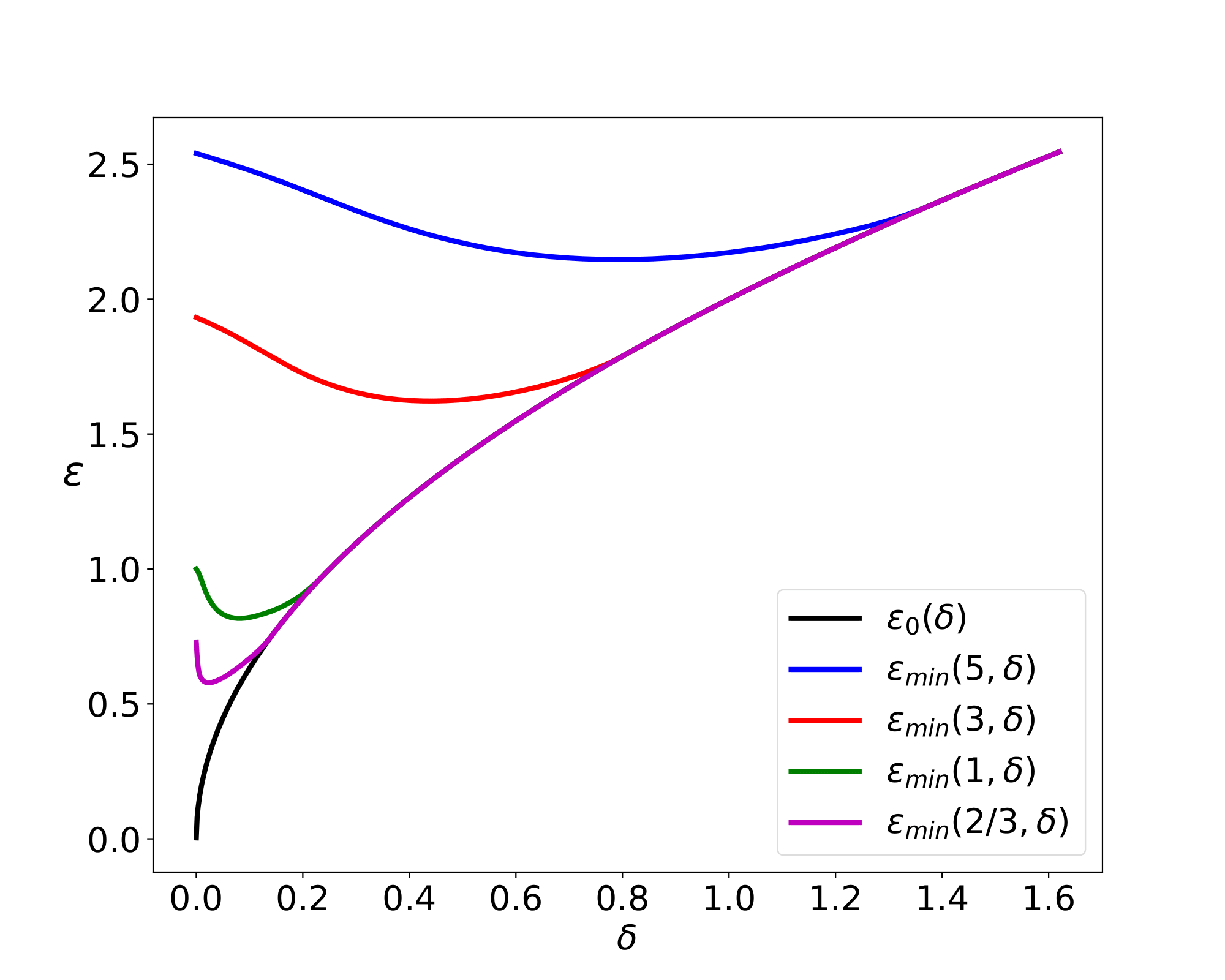}
  	\end{subfigure}
  	\caption{ $\epsilon_{min}(\alpha, \delta)$ for $\alpha >1/2$.} \label{fig5}
\end{figure}


\begin{theorem}\label{thint}
  Let $\alpha>1/2$ and
  $ \epsilon
    <\epsilon_{min}(\alpha)$ fixed, there is a value $ \delta_{max}(\alpha, \epsilon)>0$ such
    that monotone travelling waves of (\ref{tw2}) do not exist for
    $0 <\delta < \delta_{max}(\alpha, \epsilon)$.
\end{theorem}

\begin{proof} The theorem is proved if we prove that there are no
  monotone heteroclinic connections between $(u_-,0)$ and $(u_+,0)$
  for (\ref{twsys}) for $\delta$ small enough.

  In (\ref{sp}), consider the singular set corresponding to $\delta =0$. If
  $\epsilon< \epsilon_{min}(\alpha)$, this set has two components, in
  each of which $v$ can be written as a function of $w$, $v= H_\pm(w)$,
  functions that are  located on each side of the line $V= \overline{u}$,
  where $g(\overline{u}) = \min_{u\in [u_+, u_-]} g(u)$, and such that
\[
  \lim_{w \rightarrow -\infty} H_\pm(w) = u_\pm,
\]
see Figure \ref{fig6}.

 \begin{figure}[H]
    \centerline{\includegraphics[width=0.6\textwidth]{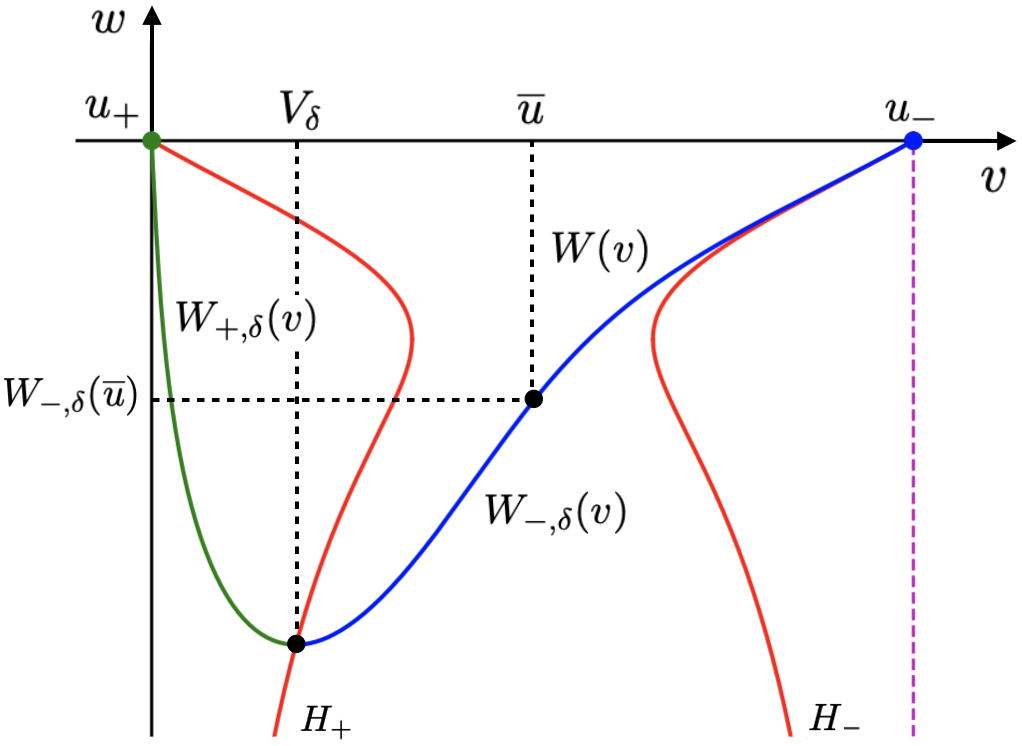}}
     \caption{$\alpha>1/2$ and $ \epsilon<\epsilon_{min}(\alpha)$.} \label{fig6}
 \end{figure} 

Now consider the portion of the unstable manifold of $(u_-, 0)$
in the domain $D$,
\[
  D = \{ (v, w) \in \R^2 \, | \, u_+ \leq v \leq u_-, w \leq 0 \},
\]
and call it $\Delta_\delta$. We see that the solution through an
initial condition on $\Delta_\delta$ travels to the left in the $v$
direction and down in the $w$ direction till we reach an intersection
with $H_+(w)$ at a point $V_\delta$. Hence we can identify
$\Delta_\delta$ with functions of $v$, which we call $W_{-,\delta}(v)$,
defined on  the interval $[V_\delta, u_-]$. Note 
that using the argument leading to (\ref{feq}), 
$W_{-,\delta}$  satisfies the
equation
\begin{equation}\label{wneg}
  \delta W_{-,\delta}\frac{dW_{-,\delta}}{dv} = g(v) - \frac{\epsilon
    W_{-,\delta}}{(1+W_{-,\delta}^2)^\alpha}.
\end{equation}

By considering the flow generated by (\ref{sp}), we have
that the functions $W_{-,\delta}(v)$ are ordered: if $\delta_1 < \delta_2$,
$W_{-,\delta_2}$ lies above $W_{-,\delta_1}$ in $S$. 

By the above argument, $W_{-,\delta}(v)$ must intersect the line
$v= \overline{u}$ at some point
$(\overline{u}, W_{-,\delta}(\overline{u}))$. Furthermore,
$W_{-,\delta_1}(\overline{u}) < W_{-,\delta_2}(\overline{u})$ if
$\delta_1< \delta_2$. We also have that
\[
  \lim_{\delta \rightarrow 0} W_{-,\delta}(\overline{u}) = -\infty.
\]

To show that, first of all note that for all
$v \in [\overline{u}, u_-]$ the right-hand side of (\ref{wneg}) is
negative. Integrating (\ref{wneg}) in $v$ between $\overline{u}$ and
$u_-$ and noting that $W_{-,\delta}(u_-)=0$, we have
\begin{equation}\label{weq2}
  \frac{\delta}{2}W_{-,\delta}(\overline{u})^2 =
  \int_{\overline{u}}^{u_-} \left[ \epsilon \frac{
    W_{-,\delta}(v)}{(1+W_{-,\delta}(v)^2)^\alpha} -g(v)\right] \, dv,
\end{equation}
with the integrand being positive. Now assume that
\[
  \lim_{\delta \rightarrow 0} W_{-,\delta}(\overline{u}) = W_0 > -\infty.
\]
Then we can take the limit of both sides of (\ref{weq2}) as
$\delta \rightarrow 0$, with the result that there must exist points
$v$ in any neighbourhood of $\overline{u}$ for which there exist pairs
$(v,w)$ in the singular set (\ref{sp}) with $\delta=0$, which is
impossible as $\epsilon< \epsilon_{min}(\alpha)$.

Since $W_{-,\delta}(\overline{u}) \rightarrow -\infty$ as
$\delta \rightarrow 0$, it follows that the $w$-coordinates of points
of intersection of $W_{-,\delta}$ with $H_+(w)$ do as well as
$\delta \rightarrow 0$; we have already denoted the $v$-coordinates of
the intersection by $V_\delta$; clearly
\[
  \lim_{\delta \rightarrow 0} V_\delta= u_+.
\]

From (\ref{wneg}), integrating in $v$ from $V_\delta$ to
$\overline{u}$, we have
\begin{equation}\label{wquad}
  \frac{\delta}{2} \left(W_{-,\delta}^2(\overline{u}) - W_{-,\delta}^2
    (V_\delta) \right) = \int_{V_\delta}^{\overline{u}} g(v) \, dv - \epsilon 
  \int_{V_\delta}^{\overline{u}}
  \frac{W_{-,\delta}(v)}{(1+W_{-,\delta}^2(v))^\alpha}\, dv.  
\end{equation}
The first integral in the right-hand side is bounded by a negative constant,
while the second satisfies

\begin{align*}  
  & \left|\int_{V_\delta}^{\overline{u}}
  \frac{W_{-,\delta}(v)}{(1+W_{-,\delta}^2(v))^\alpha}\, dv\right| \leq 
  \int_{V_\delta}^{\overline{u}}
  \frac{|W_{-,\delta}(v)|}{(1+W_{-,\delta}^2(v))^\alpha}\, dv \leq  
  \int_{V_\delta}^{\overline{u}} |W_{-,\delta}|^{1-2\alpha}\, dv\\
  & \leq
  |W_{-,\delta}^{1-2\alpha}(\overline{u})|(\overline{u}-V_{\delta}),
\end{align*}  

so that it goes to zero as $\delta \rightarrow 0$ since $\alpha>1/2$. 

Hence (\ref{wquad})  implies that there is a positive constant $C_1$ and
a value $\delta_1 >0$ such that for all $0 <\delta< \delta_1$ we have
\begin{equation}\label{West}
    W_{-,\delta}^2 (V_\delta) >  \frac{C_1}{\delta}.
  \end{equation}

  From (\ref{West}) we can get an estimate of $V_\delta$ in terms of
  $W_{-,\delta}(V_\delta)$.

  As we have
  \[
    -g(V_\delta) = \frac{\epsilon
    \lvert
    W_{-,\delta}(V_\delta)\lvert}{(1+W_{-,\delta}^2(V_\delta))^\alpha}\leq
  \epsilon \lvert W_{-,\delta}(V_\delta)\lvert^{1 - 2\alpha},
\]
using (\ref{West}), we have that for  $0 < \delta < \delta_1$  for
some positive constant $C_2$  
\begin{equation}\label{Vg1}
   -g(V_\delta)  \leq C_2\delta^{(2\alpha-1)/2}. 
\end{equation}

Hence, since 
\[
  \lim_{\delta \to 0} \frac{g(V_\delta)}{g'(u_+)(V_\delta - u_+)} = 1,
\]  
there exists a positive constant $\delta_2 \leq \delta_1$ such that for all $0 < \delta < \delta_2$,
\[
  -g(V_\delta) \geq  - \frac{1}{2}g'(u_+)(V_\delta - u_+),
\] 
and we have that there is a positive constant $C_3$ such that 
\begin{equation}\label{Vg2}
  V_\delta-u_+ \leq C_3 \delta^{(2\alpha-1)/2} .
\end{equation}

Next we show that this estimate of $V_\delta$ as $\delta \rightarrow
0$ implies that there cannot be heteroclinic solutions of
(\ref{sp}) for small enough $\delta$. For this, we will proceed by contradiction.

Assume  that there is a monotone travelling wave $W_{\delta}(v)$
connecting $(u_-, 0)$ to $(u_+, 0)$. Let us call
$W_{+,\delta}(v)$ the semi-orbit defined on $[u_+,V_\delta]$. Thus $W_{\delta}(v)= W_{-, \delta}(v)$ on $[V_\delta, u_-]$
and $W_{\delta}(v)= W_{+, \delta}(v)$ on $[u_+, V_\delta]$ with
$W_{+,\delta}(V_\delta)= W_{-,\delta}(V_\delta).$

 So $W_{+,\delta}(v)$  satisfies the equivalent 
of (\ref{tw2}) on $[u_+,V_\delta]$:

\begin{equation}\label{wpos}
  \delta W_{+,\delta}\frac{dW_{+,\delta}}{dv} = g(v) - \frac{\epsilon
    W_{+,\delta}}{(1+W_{+,\delta}^2)^\alpha}.
\end{equation}

Since on $[u_+,u_-]$, $g(v)\leq 0$ as is $ W_{+,\delta}$, rearranging,
we
have
\[
  -\delta(1+W_{+,\delta}^2)^\alpha \frac{d\vert W_{+,\delta}\lvert}{dv} + \epsilon \geq
    0.
\]    
Therefore, as between $u_+$ and $V_\delta$, $d|W_{+,\delta}|/dv > 0$, 
\[
   \delta(|W_{+,\delta}|^2)^\alpha \frac{d|W_{+,\delta}|}{dv} \leq
   \epsilon,
 \]
 and hence, by integrating over $[u_+, V_\delta]$, we obtain that for
 some $C_4>0$,
 \begin{equation}\label{Wpest}
   |W_{+,\delta}(V_\delta)|^{2\alpha+1} \leq \frac{C_4}{\delta}(V_\delta-u_+).
 \end{equation} 
  
\noindent 
{\bf Remark:} Note that this type of estimate is true for any value of $\epsilon$.

\noindent 
Using (\ref{Vg2}), we have from (\ref{Wpest}) that for all $0 < \delta < \delta_2$,
\begin{equation}\label{Wpest2}
  |W_{+,\delta}(V_\delta)|^{2\alpha+1} \leq C_4 \delta^{\alpha-3/2}.
\end{equation}

Therefore from (\ref{West}) we have that for all $0 < \delta < \delta_2$ there is some positive constant $C_5$ such that
\begin{equation}\label{rat}
  1= \left( \frac{|W_{+,\delta}(V_\delta)|}{|W_{-,\delta}(V_\delta)|}
  \right)^{2\alpha+1} \leq C_5 \delta^{2\alpha-1}.
\end{equation}  
But since $\alpha >1/2$ this inequality is not true for small values of $\delta$
, i.e. no monotone travelling wave can
exist. 
\end{proof}

Now we can add to Theorem~\ref{cont}: 

\begin{theorem}\label{limit0}
  Let $ \alpha > 1/2$ fixed. Then 
  \[
 \lim_{\delta \rightarrow 0}\epsilon_{min} (\alpha, \delta) =  \epsilon_{min}(\alpha).
\]
\end{theorem}

\begin{proof} 
  Let $A>0$ and
  $\epsilon_1 = \epsilon_{min}(\alpha) -A < \epsilon_{min}(\alpha)$.
  By Theorem~\ref{thint}, there is
  $ \delta_{max}(\alpha, \epsilon_1)>0$ such that monotone travelling
  waves of (\ref{tw2}) do not exist for
  $0 <\delta < \delta_{max}(\alpha, \epsilon_1)$. So
  $\epsilon_{min}(\alpha, \delta) > \epsilon_1$.
 
  Now, let
  $\epsilon_2 = \epsilon_{min}(\alpha) + A >\epsilon_{min}(\alpha)$.
  By Theorem~\ref{smdelta}, there is
  $ \delta_{0}(\alpha, \epsilon_2)>0$ such that monotone travelling
  waves of (\ref{tw2}) exists and enters $(u_+, 0)$ through the main
  manifold for $0 <\delta < \delta_{0}(\alpha, \epsilon_2)$. So
  $\epsilon_min(\alpha, \delta) < \epsilon_2$.
    
    Finally for all
    $0 < \delta < \min(\delta_{0}(\alpha, \epsilon_2);
    \delta_{max}(\alpha, \epsilon_1))$,
    $ \vert \epsilon_{min} (\alpha, \delta) - \epsilon_{min}(\alpha)
    \vert < A$.
  \end{proof}

Note that this Theorem is not true for $\alpha=1/2$; see~\ref{three}
for a further discussion of this case.
  
From the above results, we immediately have   
\begin{corollary}\label{nld2}
  For $\alpha \geq 1/2$ and $\delta$ sufficiently small, (\ref{tw2})
  is nonlinearly determined.
\end{corollary}

We have also the following statements:

\begin{lemma}\label{md}
For $\alpha>1/2$ the function $\delta \mapsto
\epsilon_{min}(\alpha,\delta)$ cannot be monotone decreasing.
\end{lemma}

\begin{proof}
  If this were the case, we would contradict Theorem~\ref{theolin}:
  for $\delta$ sufficiently large, the graph of 
  $\delta \mapsto \epsilon_{min}(\alpha,\delta)$ must be tangent to
  that of $\epsilon_0(\delta)$ which is monotone increasing as the two
  graphs cannot intersect transversally.
\end{proof}

\begin{lemma}\label{mi}
For $\alpha>1/2$ the function $\delta \mapsto
\epsilon_{min}(\alpha,\delta)$ cannot be monotone increasing.
\end{lemma}

\begin{proof}
This result follows by Corollary~\ref{epsineq}.  
\end{proof}

\section{Remarks and Conclusions}

In this paper we have conducted a detailed analysis of the existence
of monotone travelling waves for the generalised
Rosenau-KdV equation (\ref{Ros1}). Our results give a complete
description of the domain of existence of such waves in the case
$\alpha<1/2$, while in the case of $\alpha > 1/2$ what is missing is a
proof of the (conjectured) uniqueness of the minimum of
$\epsilon_{min}(\alpha, \delta)$. The borderline case $\alpha=1/2$ is
of particular interest. In this case the domain of existence of
monotone travelling waves is rather complicated. While
Theorem~\ref{theoMondelta} provides a monotonicity result in that
case, we also have (see~\ref{three}) that
$\lim_{\delta \rightarrow 0} \epsilon_{min}(1/2,\delta)=3/8$, while
$\lim_{\alpha \rightarrow 1/2_+} \epsilon_{min}(\alpha)=1/2$.

Following Schonbek \cite{Schonbek1982} in the case $\alpha =0$, and
the result of Lemma \ref{c0} we formulate the following conjecture:

\noindent
{\bf Conjecture:} Solution of Cauchy problems for the
Rosenau-KdV equation converge to the entropy solution of (\ref{iB}) if
$\delta = o(\epsilon^{1/(1/2 - \alpha)})$ as $\epsilon \to 0$ in the
case $\alpha < 1/2$.  When $\alpha \geq 1/2$ solution of the Cauchy
problem do not converge to the entropy solution of (\ref{iB}) as
$\epsilon$ and $\delta$ go to $0$.

Our analysis leaves open a number of interesting problems:
\begin{itemize}
\item We have shown that in every case of $\alpha>0$, the travelling
  wave problem (\ref{tw2}) is nonlinearly determined for $\delta$
  small enough and linearly determined for $\delta$ large enough. A
  better characterisation of the minimality transition point than that
  given in~\ref{epst} is clearly required.

\item We did not prove the uniqueness of minimum for
  $\epsilon_{min}(\alpha,\delta)$  for $\alpha>1/2$. Perhaps Melnikov
  methods as used for example in Ref.\cite{Balasuriya} would be useful
  here.

\item In the threshold case $\alpha=1/2$, we gave an argument for
  locating the value of
  $\lim_{\delta \rightarrow 0_+} \epsilon_{min}(1/2,\delta)$. Let us
  denote this value of $\epsilon$ by $\tau_0$, $0$ indicating that the
  curve connecting to it is the boundary of the region on the
  $(\delta, \epsilon)$ plane of heteroclinic connections with zero
  intersections of the $v=u_+$ line. It would also be interesting to
  determine the location of the points $\epsilon=\tau_k$, $k \geq 1$,
  lying on the boundary of regions of existence of heteroclinic
  connections with $k$ intersections with the line  $v=u_+$ . 

\item Clearly, there are PDE-theoretic aspects of the problem that
  were outside the scope of the present study. A global existence
  theory for the Rosenau-KdV is needed, as is, in general, a 
  characterisation of properties of the generalised Rosenau diffusion
  operator, especially in the case of $\alpha > 1/2$. 
\end{itemize}

Finally we note that it would also be interesting to consider the
existence of monotone travelling waves for the following (generalised)
Burgers-Rosenau equation 
 \begin{equation}\label{Ros2}
   u_t+f(u)_x= \epsilon u_{xx} +
   \delta\left(\frac{u_x}{(1+u_x^2)^\alpha}\right)_{xx},
\end{equation}
in which case the papers of Garrioni and 
coauthors \cite{GarrioneSanchez2015,GarrioneStrani2019} will be useful for the case $\alpha=1/2$.

\appendix

\section{Hadeler--Rothe computation of $H(\alpha)$} \label{HR}

In this appendix we use the ideas of Hadeler and
Rothe \cite{HadelerRothe75} to arrive exactly at (\ref{emin2}).

From
(\ref{feq}) we have that
\[
  \epsilon_{min}(\alpha,\delta) \geq \sup_{v \in [u_+,u_-]} \left[
    \delta F'(v) - \frac{g(v)}{F(v)}
    \right]\left[1+(F(v))^2\right]^\alpha.
\]
Now, we use the Hadeler-Rothe type approximation, approximating
$F(v)$  by $-Ag(v)$ and have the approximation scheme
\begin{equation}\label{appr1}
  \epsilon_{min}(\alpha,\delta) \approx \inf_{A \in \R_+} \sup_{v \in [u_+,u_-]}
  \left[ -\delta Ag'(v) + \frac{1}{A}
  \right]\left[1+A^2 (g(v))^2\right]^\alpha.
\end{equation}
As $\delta \rightarrow 0$, the infimum in $A \in \R_+$ is reached for
a finite value of $A$ if $\alpha>1/2$. To take the limit in the
right-hand side of (\ref{appr1}) as $\delta \rightarrow 0$, or
alternatively, to find the leading term in the asymptotic expansion of
our approximation of $\epsilon_{min}(\alpha, \delta)$ in powers of
$\delta$, we can just disregard the $-\delta Ag'(v)$ term to obtain
\begin{equation}\label{appr2}
  \epsilon_{min}(\alpha) \approx
  \inf_{A \in \R_+} \sup_{v \in [0,2]} \frac{1}{A} \left[1+A^2
    (g(v))^2\right]^\alpha. 
\end{equation}
Now let us choose
$u_-=2$, $u_+=0$, so that $g(v)=\frac12 v(v-2)$.  With that choice,
which also leads to $S=1/2$, we have 
\begin{equation}\label{appr3}
  \epsilon_{min}(\alpha) \approx
  \inf_{A \in \R_+} \sup_{v \in [0,2]} \frac{1}{A} \left[1+\frac{A^2}{4}
    v^2(2-v)^2\right]^\alpha. 
\end{equation}
The supremum in $v$ is reached at $v=1$, so we have the very elegant
formula,
\begin{equation}\label{20a}
  \epsilon_{min}(\alpha)
  \approx \inf_{A \in \R_+} \frac{\left[1+A^2/4\right]^\alpha}{A}.
\end{equation}
This minimisation problem can be solved explicitly and we finally
obtain that
\begin{equation}\label{20b}
  \epsilon_{min}(\alpha)
  \approx 2^{-1+\alpha} \left( \frac{\alpha}{2\alpha-1}
  \right)^\alpha \sqrt{2\alpha-1},
\end{equation}
which is {\em exactly} the formula (\ref{emin2}) obtained in
Lemma~\ref{epsmin} by the tangency argument.

\section{Computation of $\lim_{\delta \rightarrow 0_+}
  \epsilon_{min}(1/2,\delta)$} \label{three}

We want to establish that if $u_+=0$, $u_-=2$, and hence $g(v)=
v(v-2)/2$, 
\[
  \lim_{\delta \rightarrow 0_+} \epsilon_{min}(1/2,\delta) = \frac38.
\]

First of all however we want to discuss why this result does not
contradict the result that follows from Lemma~\ref{epsmin} that shows that
in this case
\[
  \lim_{\alpha \rightarrow 1/2_+} \epsilon_{min} (\alpha)= \frac12.
\]

The reason is that for $\alpha=1/2$ the mapping $\delta \mapsto
\epsilon_{min}(1/2,\delta)$ is not continuous at $\delta=0$.

So let us fix $\alpha=1/2$ and take $\delta$ small. In Figure~\ref{pp}
we show the heteroclinic orbit for the case of $u_+=0$, $u_-=2$, and
$g(v)= v(v-2)/2$, with small $\delta$ and $\epsilon \approx 3/8$. From
that picture it looks that the heteroclinic can be broken in two
pieces, one close to the saddle point at $(u_-,0)$, and the rest. We
will find the leading term of the expansion (in $\sqrt{\delta}$ of the
portion of the (closure of) heteroclinic orbit that satisfies the zero
boundary condition at $(u_+,0)$.

  \begin{figure}[H]
    	\centerline{\includegraphics[width=0.7\textwidth]{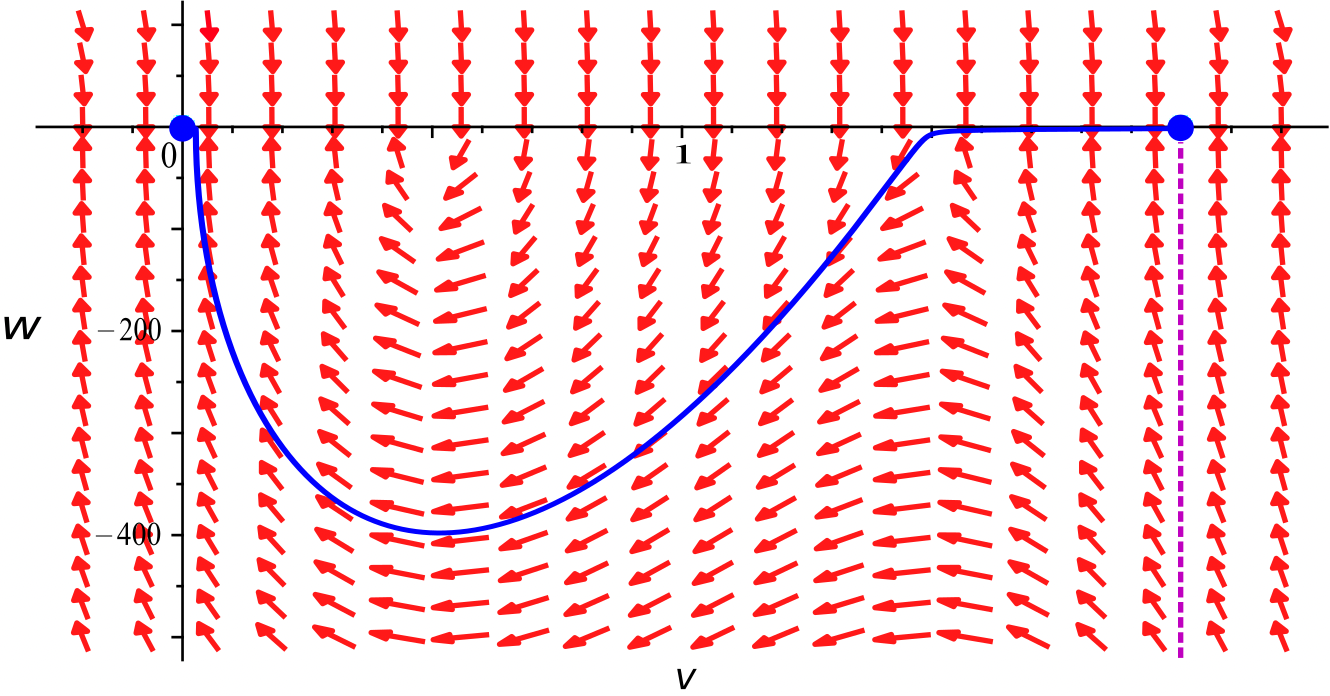}}
     	\caption{The structure of the heteroclinic orbit.} \label{pp}
  \end{figure} 
  
In general, we are solving the second order system
\[
  \delta v'' + \frac{\epsilon v'}{\sqrt{1+(v')^2}} - g(v) =0.
\]

Setting as in (\ref{feq}), $v'= -F(v)$ and using the chain rule, we have
\[
  \delta F F' = \frac{\epsilon F}{\sqrt{1+F^2}} + g(v). 
\]
Now put $F=Z/\sqrt{\delta}$. Then
\[
  \frac{d}{dv} \left(\frac{1}{2} Z^2 \right) =
  \frac{\epsilon Z}{\sqrt{\delta+Z^2}} + g(v).
\]
We seek the expansion $Z(v) \sim Z_0(v)+ O(\sqrt{\delta})$.  Note that
$Z\geq 0$, so $Z=|Z|$. For $Z_0(v)$ we
then have the trivial equation
\begin{equation}\label{lt}
  \frac{d}{dv} \left(\frac{1}{2} Z_0^2 \right) = \epsilon + g(v). 
\end{equation}

However, $Z_0$ has to satisfy the two boundary conditions $Z_0(u_+)=0$ and
$Z_0(u_-)=0$.

We see that the boundary condition at $v=u_-$ cannot be satisfied, so
that we use $Z_0(u_+)=0$.  Integrating (\ref{lt}), we obtain
\[
 Z_0(v) := \sqrt{2}\sqrt{G(v)+\epsilon (v-u_+)},
\]

where $G(v) = {\displaystyle \int_{u_+}^v g(v)\, dv}$. By properties
of $g$, $G$ is a monotone decreasing function, $G(u_+)=0$. We now choose
$\epsilon$ which solves the problem 
\[
  \hbox{argmin}_{\epsilon \geq 0} \{ Z_0(v) \hbox{ is defined for all }
  v \in [u_+,u_-]\}.
\]
Such a value of $\epsilon$ exists and is defined by
\[
 \epsilon := \max_{ v \in [u_+,u_-]}\frac{1}{(v-u_+)} \vert G(v)\vert
 = \max_{ v \in [u_+,u_-]}\left(\frac{1}{(v-u_+)}{\displaystyle
   \int_{u_+}^v \vert g(v)\vert\, dv}\right).
\]

We specialise now to our example, $u_+=0$, $u_-=2$, and
$g(v)= v(v-2)/2$, and see that this selection criterion indeed gives
us the value $\epsilon=3/8$.

We make a number of remarks.

1. This selection criterion makes sense as we are looking at the limit
as $\delta \rightarrow 0$ of values of $\epsilon$ for which there exist
monotone travelling waves; clearly, if such solutions exist, $Z_0(v)$
will capture the leading term behaviour close to the point
$(u_+,0)$. Hence if $Z_0(v)$ cannot be constructed, it is reasonable to
assume that $\epsilon$ is too small.

2. Numerically, $-F(v) \approx Z_0(v)/\sqrt{\delta}$ gives a very good
approximation to the heteroclinic for small $\delta$. For
$\epsilon=3/8$, the maximum of $Z_0(v)$ is achieved at $v=1/2$ and 
the minimum of $-F(v)$ is approximately given by
\[
  \min_{v \in [0,2]} -F(v) \approx \frac{-0.40827}{\sqrt{\delta}},
\]
which is well borne out by numerics for small $\delta$.

3. To continue the expansion, we write
\[
  Z(v) \sim Z_0(v) + \sqrt{\delta} Z_1(v) + \cdots
\]
(it is probably easier to work with $Z^2(v)$) and
\[
  \epsilon=3/8 + \sqrt{\delta} \epsilon_1 + \cdots,
\]
which is the dependence on $\delta$ seen in numerics. {\em A priori}
it is not clear what is the selection criterion for $\epsilon_1$. 

\section{Lower bound for $\tilde{\epsilon}(\alpha)$ and
  $\delta_\alpha$ for $\alpha >1/2$} \label{epst}

Set $\alpha>1/2$. Then define

\[
  A(\alpha) = \max_{z \in {\mathbb R}_+} \frac{z}{(1+z^2)^\alpha} = {(2\alpha)}^{-\alpha}(2\alpha -1)^{\alpha-1/2} .
\]

Then
\[
  \epsilon A(\alpha) \geq \epsilon \frac{F(v)}{(1+F(v)^2)^\alpha} =
  \delta F(v)F'(v)-g(v).
\]

Now integrate over $[u_+, v]$, $v\in (u_+, u_-]$ and we have
\[
  \tilde{\epsilon} (\alpha) 
  A(\alpha) \geq 
  \frac{1}{(v-u_+)}
\int_{u_+}^{v} | g(v)| \, dv.
\]

This is true for all $v\in (u_+, u_-]$, thus

\[
  \tilde{\epsilon} (\alpha) \geq \frac{1}{A(\alpha)}
  \max_{v\in [u_+, u_-]}
  \left(\frac{1}{(v-u_+)}
  \int_{u_+}^{v} | g(v)| \, dv\right) =\mu(\alpha).
\]

As $A(\alpha)$ is monotone decreasing in $\alpha$, the estimate is
monotone increasing in $\alpha$.
  
It is quite a conservative estimate, e.g. for $\alpha=1$ it gives 3/4,
while the correct value is about 0.8.

This estimate gives a
bound on the minimality exchange: a lower bound for $\delta_\alpha$
for $\alpha > 1/2$ is 
\[
  \delta_\alpha \geq \frac{\mu(\alpha)^2}{4}.
\]  

\section*{Acknowledgements}

GM gratefully acknowledges support from the ISP (Uppsala University,
Sweden); JMCC would like to thank EMS for funds that made possible a
visit to Glasgow where some of the work presented here was done; 
JMCC and MG acknowledge support of FCT, Portugal, under CIMA project
UIDB/04674/2020, DOI: 10.54499/UIDB/04674/2020.


\noindent
%
%
%
%
%

(Gnord Maypaokha) LAMFA-CNRS UMR 7352, Universit\'e de Picardie Jules Verne, France \& Department of Mathematics and Statistics, National University of Laos, Lao PDR\\
\textit{Email address}: \texttt{{gnord\_1685@yahoo.com}}\\

(Nabil Bedjaoui) LAMFA-CNRS UMR 7352, Universit\'e de Picardie Jules Verne, France \\
\textit{Email address}: \texttt{{nabil.bedjaoui@u-picardie.fr}}\\

(Joaquim M. C. Correia) CIMA, IIFA, U\'Evora, \'Evora, Portugal \&
CAMGSD, IST, Lisboa, Portugal\\
\textit{Email address}: \texttt{{jmcorreia@uevora.pt}}\\

(Michael Grinfeld) Department of Mathematics and Statistics, University of Strathclyde, Glasgow G1 1XH, UK\\
\textit{Email address}: \texttt{{m.grinfeld@strath.ac.uk}}\\

\end{document}